\documentclass[leqno,11pt,a4paper]{article}
\usepackage[left=2.0cm, top=2.2cm,bottom=2.5cm,right=2.0cm]{geometry}
\usepackage{amsmath,amssymb,amsthm,mathrsfs,calc,graphicx}
\usepackage[british]{babel}
\numberwithin{equation}{section}

\newtheorem {theorem}{Theorem}[section]
\newtheorem {proposition}[theorem]{Proposition}
\newtheorem {lemma}[theorem]{Lemma}
\newtheorem {remark}[theorem]{Remark}
\newtheorem {corollary}[theorem]{Corollary}

\newcommand{\N}{\mathbb{N}}
\newcommand{\R}{\mathbb{R}}

\def\ba{\begin{array}}
\def\ea{\end{array}}
\def\bea{\begin{eqnarray} \label}
\def\eea{\end{eqnarray}}
\def\be{\begin{equation} \label}
\def\ee{\end{equation}}
\def\bit{\begin{itemize}}
\def\eit{\end{itemize}}
\def\ben{\begin{enumerate}}
\def\een{\end{enumerate}}

\def\BB{\mathbb{B}}

\def\EE{\mathbb{E}}

\def\NN{\mathbb{N}}
\def\PP{\mathbb{P}}

\def\RR{\mathbb{R}}
\def\SS{\mathbb{S}}

\def\a{\alpha}
\def\b{\beta}
\def\g{\gamma}
\def\d{\delta}
\def\e{\varepsilon}

\def\k{\kappa}
\def\l{\lambda}
\def\r{\varrho}

\def\o{\omega}
\def\G{\Gamma}

\def\cF{\mathcal{F}}

\def\cH{\mathcal{H}}
\def\cI{\mathcal{I}}

\def\dint{\textup{d}}
\def\tW{\widetilde{V}}
\def\var{{\textup{var}}}

\allowdisplaybreaks

\begin{document}

\title{\bfseries Poisson polyhedra in high dimensions}

\author{Julia H\"orrmann\footnotemark[1], Daniel Hug\footnotemark[2], Matthias Reitzner\footnotemark[3] and Christoph Th\"ale\footnotemark[4]}

\date{\today}
\renewcommand{\thefootnote}{\fnsymbol{footnote}}

\footnotetext[1]{Karlsruhe Institute of Technology, Department of Mathematics, Institute of Stochastics, D-76128 Karlsruhe, Germany. E-mail: julia.hoerrmann@kit.edu}

\footnotetext[2]{Karlsruhe Institute of Technology, Department of Mathematics, Institute of Stochastics, D-76128 Karlsruhe, Germany. E-mail: daniel.hug@kit.edu}

\footnotetext[3]{Osnabr\"uck University, Institute of Mathematics, Albrechtstra\ss e 28a, D-49076 Osnabr\"uck, Germany. E-mail: matthias.reitzner@uos.de}

\footnotetext[4]{Ruhr University Bochum, Faculty of Mathematics, NA 3/68, D-44780 Bochum, Germany. E-mail: christoph.thaele@rub.de}

\maketitle

\begin{abstract}
The zero cell of a parametric class of random hyperplane tessellations depending on a distance exponent and an intensity parameter is investigated, as the space dimension tends to infinity. The model includes the zero cell of stationary and isotropic Poisson hyperplane tessellations as well as the typical cell of a stationary Poisson Voronoi tessellation as special cases. It is shown that asymptotically in the space dimension, with overwhelming probability these cells satisfy the hyperplane conjecture, if the distance exponent and the intensity parameter are suitably chosen dimension-dependent functions. Also the high dimensional limits of the mean number of faces are explored and the asymptotic behaviour of an isoperimetric ratio is analysed. In the background are new identities linking the $f$-vector of the zero cell to certain dual intrinsic volumes.
\end{abstract}

\section{Introduction}

Over the past two decades, the theory of random polytopes and random polyhedra has advanced significantly. This development has been driven by new geometric and probabilistic techniques for establishing asymptotic results for random polytopes, but also by various connections and applications to other branches of mathematics such as optimization \cite{Borgwardt}, convex geometric analysis \cite{KK}, extreme value theory \cite{LaoMayer,CC,Chenavier}, multivariate statistics \cite{Cascos}, random matrices \cite{LPRT}, and algorithmic geometry \cite{Chazelle,Preparata}. We refer to the survey articles \cite{HugSurvey} and \cite{ReitznerSurvey} for more background material and further references.

Random constructions and basic probabilistic reasoning often provide the existence of an object with desirable properties
which is not accessible via a purely deterministic approach. A notoriously difficult problem, to which this fundamental observation may well apply, is the hyperplane conjecture or slicing problem. In one of several equivalent formulations, it asserts the existence of a universal constant $c>0$ such that for any space dimension $n$ and for any convex body $K\subset\RR^n$ of volume one, there is a hyperplane $L$ in $\RR^n$ such that the intersection $K\cap L$ has $(n-1)$-dimensional volume at least $c$. This problem has inspired a very fruitful line of research which has been initiated by Bourgain in \cite{Bourgain1} and since then has  become one of the major open problems in geometric and functional analysis and in asymptotic convex geometry. The best lower bound for $c$ known up to date is decreasing with the space dimension $n$ and is of the order $n^{-1/4}$, due to a result of Klartag \cite{Klartag1}. For further background material we refer to works of Ball \cite{Ball1}, Junge \cite{Junge1}, Klartag and Kozma \cite{KK}, E.\ Milman \cite{MilmanEmanuelSlicing} or V.\ Milman and Pajor \cite{Milman1}, to the recent monograph \cite{Brazitikos}, as well as to the references cited therein.

The hyperplane conjecture is known to be true for special classes of convex bodies, like zonoids or dual zonoids \cite{Ball1}, unconditional convex bodies \cite{Bourgain1,Milman1} or unit balls of Schatten norms \cite{KoenigMayerPajor}, to name just a few. However, despite considerable effort over a period of now nearly 30 years, a general proof is still missing. Instead, Klartag's bound
 may even be a natural threshold.  These insights have recently led to contributions which investigate possible counterexamples, with a special focus on randomly generated polytopes (see \cite{AlonsoGutierrez.2008,DafnisGiannopGuedon,KK}). It is also worth mentioning that there are  re-formulations of the hyperplane conjecture in terms of random polytopes.
A prominent example is related to Sylvester's problem on the expected volume $\EE V_n(K,n+1)$ of a random simplex in an $n$-dimensional convex body $K$. It is known that the hyperplane conjecture would follow from the inequality $\sup_K \EE V_n(K,n+1)\leq \EE V_n(\Delta_n,n+1)$,
where $\Delta_n$ is a unit-volume simplex in $\RR^n$ and where the supremum is extended over all convex bodies $K\subset\RR^n$ with volume one (compare with \cite{Brazitikos} or with the appendix of \cite{Rademacher}).

\medspace

The major object of investigation in the present paper is a parametric class of random polyhedra. For an arbitrary space dimension $n \geq 2$, we will define an isotropic Poisson hyperplane process in $\RR^n$ which depends on a distance exponent $r>0$ and an intensity $\gamma$. This hyperplane process gives rise to a random hyperplane tessellation and thus to a system of random polyhedra, which are the cells of the  tessellation. In the focus of our attention is the cell containing the origin, the zero cell  of the tessellation, which is denoted by $Z_0$.  If for instance $r=n$, then the zero cell is equal in distribution to the typical cell of a classical Poisson-Voronoi tessellation \cite[Chapter 10.2]{SW}. In the following, we consider the normalized zero cell, which is a re-scaled version of $Z_0$ of unit volume. A special case of one of our main results (see Theorem \ref{thm:HyperplaneConjecture}) is the following theorem.

\begin{theorem}\label{th:1}
Assume that $r=b\,n^\alpha$ for some $b>0$ and $\alpha>1/2$. Then, as the space dimension $n$ tends to infinity, the probability that the hyperplane conjecture holds for the normalized zero cell tends to one.
\end{theorem}

Theorem \ref{th:1} is related to general investigations dealing with the combinatorial structure and the geometry of the zero cells $Z_0$  obtained within the class of random tessellations considered in this paper. The starting point is a set of identities connecting the number of $\ell$-dimensional faces with certain dual intrinsic volumes of $Z_0$. In the special case $r=1$, these identities reduce to a result of Schneider \cite{Schneider09} involving the well-known intrinsic volumes. It is worth mentioning that our identities are very much in the spirit of Efron's identity for random convex hulls \cite{Efron}. They provide a link between the combinatorial structure of $Z_0$ and certain metric quantities of the zero cell. In a next step, bounds for the expected dual intrinsic volumes are established and our identities are then used to obtain bounds for the expected number of faces of the random polyhedra. We also investigate the expected measure of the $\ell$-skeleton. Evaluating these bounds as $n \to \infty$ is the basis for our asymptotic results.

Besides the zero cell $Z_0$ itself, for $r=1$ we also deal with lower-dimensional weighted faces of the tessellation. Alternatively, in our setting these weighted faces can be obtained as intersections of $Z_0$ with a stochastically independent isotropic linear subspace. This interpretation is one of the motivations for studying properties of $Z_0\cap L$, where $L$ is an $m$-dimensional linear subspace of $\RR^n$. Thanks to the special structure of our tessellation model, we are able to show a transfer principle, which allows us to translate results from $Z_0$ to intersections $Z_0\cap L$ of $Z_0$ with a subspace $L$. Combined with a recent result from \cite{HoHu} this yields Theorem \ref{th:1}.

Another aspect our paper deals with, is the question of how close the random polyhedra $Z_0$ are to a Euclidean ball. More specifically, we study  how the isoperimetric ratio of mean surface area and mean volume behaves as the dimension goes to infinity. Roughly speaking, we will see, for example, that typical Poisson-Voronoi cells and their close relatives corresponding to a distance exponent $r=b\,n$ are approximately spherical in the mean, whereas the shape of the zero cell for fixed distance exponents $r$ is degenerate, in this sense. More generally, we will investigate the isoperimetric ratio for general distance exponents of the form $r=b\,n^\a$ with $b>0$ and $\a\in\RR$, see Theorem \ref{thm:IsoperimetricRatio}.

In a last step, we determine the asymptotic behaviour of $\EE f_\ell (Z_0)$, which is the mean number  of $\ell$-dimensional faces of the zero cell $Z_0$, for some fixed  $\ell\in\N_0$, as the space dimension $n$ tends to infinity.  Again, there is a remarkable difference, for example, between the case $r=b\,n$ and that of a constant distance exponent $r$. To highlight this difference, we formulate the result at this point only in the special cases where $r$ is fixed or $r$ is proportional to the dimension, and refer to Theorem \ref{thm:HighDimensionsVertices}, Theorem \ref{thm:HighDimensionRfest} and Theorem \ref{thm:HighDimensionsR=bn} below for extensions in several directions.

\begin{theorem}
Let $Z_0$ be the zero cell of a hyperplane tessellation with distance exponent $r>0$. Let $\ell\in\N_0$ be fixed. If $r$ is fixed, then
$$ \lim\limits_{n\to\infty}\sqrt[n]{\EE f_\ell(Z_0)}
=\frac{\o_{r+1}}{\k_r}\,,$$ where $\o_{r+1}$ and $\k_r$ are constants  given by \eqref{eq:OmegaKappaExakt} below.
If $r=b\,n$ for some fixed $b>0$, then
$$\lim\limits_{n\to\infty}n^{-1/2}\,\sqrt[n]{\EE f_\ell(Z_0)}=\sqrt{2\pi b}\,. $$
\end{theorem}

The paper is structured as follows. After setting up our framework together with some background material in Section \ref{sec:Preliminaries}, our basic identities are presented in Section \ref{subsec:GeneralResults}. Special formulae for $r=1$ are contained in Section \ref{subsec:SpecialCaser=1}, whereas Section \ref{subsec:PlanarSections} focuses on  sections with subspaces and the transfer principle. The hyperplane conjecture for $Z_0$ is discussed in Section \ref{subsec:asymptotikundhyperplane}, while Section \ref{subsec:HighDimensions} deals with the asymptotic behaviour of the isoperimetric ratio and with the  combinatorial structure
of the zero cell in high dimensions. The detailed proofs of our results are provided in the final Section \ref{sec:Proofs}.

\section{Preliminaries}\label{sec:Preliminaries}

\paragraph{Basic notation.}
In this paper we work in the Euclidean vector space $\RR^n$, $n\geq 2$, whose standard scalar product and induced norm will be denoted by $\langle\,\cdot\,,\,\cdot\,\rangle$ and $\|\,\cdot\,\|$, respectively. The unit sphere $\SS^{n-1}$ and the unit ball $\BB^n$ in $\RR^n$ are given by $\SS^{n-1}=\{x\in\RR^n:\|x\|=1\}$ and $\BB^n=\{x\in\RR^n:\|x\|\leq 1\}$. We write $\{e_1,\ldots,e_n\}$ for the standard basis of $\RR^n$. For a subspace $U$ of $\RR^n$ we denote by $U^\bot$ its orthogonal complement.

 Let $G(n,\ell)$ be the space of $\ell$-dimensional linear subspaces of $\RR^n$, which is equipped with the standard topology and the unique Haar probability measure $\nu_\ell$. For a subspace $L\in G(n,\ell)$ we write $\SS_L$ and $\BB_L$ for the unit sphere and the unit ball in $L$, respectively. Let us further denote by $\cH^s$, $s\geq 0$, the $s$-dimensional Hausdorff measure. By $A(n,\ell)$ we mean the space of $\ell$-dimensional affine subspaces of $\RR^n$, together with the canonical topology. On $A(n,\ell)$ we have the translation invariant measure $\mu_\ell$ defined by the relation
\begin{equation}\label{eq:DefinitionMuEll}
\int\limits_{A(n,\ell)}h(E)\,\mu_\ell(\dint E)=\int\limits_{G(n,\ell)}\int\limits_{L^\perp}h(L+x)\,\cH^{n-\ell}(\dint x)\,\nu_\ell(\dint L)\,,
\end{equation}
where $h\geq 0$ is a measurable function on $A(n,\ell)$ (in a topological space, measurability in this paper always refers to the
 Borel $\sigma$-field). If $L\in G(n,\ell)$, for some $\ell\in\{0,1,\ldots,n-1\}$, we write $\k_\ell=\cH^\ell(\BB_L)$ for the $\ell$-volume of $\BB_L$, and we write $\o_\ell=\cH^{\ell-1}(\SS_L)$ for the $(\ell-1)$-volume of $\SS_L$ if $\ell \in\{1,\ldots,n-1\}$. It is well known that
\begin{equation}\label{eq:OmegaKappaExakt}
\o_\ell=\frac{2\pi^{\frac \ell 2}}{\Gamma\big({\frac \ell 2}\big)}
\qquad\text{and}\qquad
\k_\ell=\frac {\pi^{\frac \ell 2} }{ \Gamma\big({\frac \ell 2}+1\big)}\,,
\end{equation}
where $\Gamma(\,\cdot\,)$ denotes the gamma function.
 We will keep the notation $\o_r$ and $\k_r$ as shorthand for \eqref{eq:OmegaKappaExakt} also for real-valued parameters $r\geq 0$. We repeatedly use the relation
\begin{equation}\label{eqGamma}
\int\limits_0^{\frac{\pi}{2}}(\sin\varphi)^\alpha (\cos\varphi)^\beta d\varphi = \frac{\o_{\alpha+\beta+2}}{\o_{\a+1}\o_{\b+1}},
\end{equation}
for $\alpha,\beta>-1$; see \cite[Equation (5.6)]{Artin1} or \cite[Equation (12.42)]{WW}.
 For $u_1,\ldots,u_m\in\RR^n$, we define $\nabla_m(u_1,\ldots,u_m)$ as the $m$-volume of the parallelepiped spanned by the vectors $u_1,\ldots,u_m$. Moreover, for a set $K\subset\RR^n$ and a subspace $L\in G(n,\ell)$, we write $K|L$ for the orthogonal projection of $K$ to $L$.
For a polytope $P$ and $j\in\{0,\ldots,n\}$ we denote by $\cF_j(P)$ the set of $j$-dimensional faces of $P$. For $x\in\RR$ the positive part of $x$ is  $x_+=\max\{x,0\}$.

\paragraph{Intrinsic volumes.}
For a convex body (a non-empty, compact and convex subset) $K\subset\RR^n$ and arbitrary $s>0$ we consider the volume $\cH^n(K_s)$ of the $s$-parallel set $K_s=\{x\in\RR^n:d(x,K)\leq s\}$ of $K$, where $d(x,K)=\inf\{\|x-y\|:y\in K\}$ is the distance between $x$ and $K$. According to Steiner's formula \cite[Equation (4.1)]{Schneider14}, $\cH^n(K_s)$ is a polynomial in $s$ of degree $n$. Thus, there are constants $V_0(K),\ldots,V_n(K)$, the  intrinsic volumes of $K$, such that
$$\cH^n(K_s)=\sum_{j=0}^n s^{n-j}\k_{n-j}V_j(K)\,,\quad s\geq 0\,.$$
In particular, if $0<\cH^n(K)<\infty$, then $V_n(K)=\cH^n(K)$, $V_{n-1}(K)$ is half of the $(n-1)$-dimensional surface area of $K$ and $V_0(K)= 1$. It is a particular feature of the normalization of the intrinsic volumes that they do not depend on the dimension of the surrounding space. In other words, if a convex body $K\subset\RR^n$ is contained in some lower-dimensional subspace $L\in G(n,m)$ of dimension $m\in\{1,\ldots,n-1\}$, then $V_j(K)=0$ for $j>m$ and for $j\leq m$, $V_j(K)$ evaluated in $\RR^n$ yields the same result as $V_j(K)$ evaluated within $L$.

\paragraph{Dual intrinsic volumes.}
Let $s\in\RR$ and let $K\subset\RR^n$ be a convex body with $0\in K$ and radial function $\r_K(u)=\max\{\l:\l u\in K\}$ in direction $u\in\SS^{n-1}$. We call
$$\tW_s(K)={\frac 1 n}\int\limits_{\SS^{n-1}}\r_K(u)^{n-s}\,\cH^{n-1}(\dint u)$$
the dual intrinsic volume of order $n-s$ of $K$, see \cite{Lutwak75}. If $L\in G(n,\ell)$ and $K\subset L$, we define
$$\tW^L_s(K)={\frac 1 \ell}\int\limits_{\SS_L}\r_K(u)^{\ell-s}\,\cH^{\ell-1}(\dint u)$$
as the dual intrinsic volume of order $\ell-s$ with respect to $L$. We stress the fact that the dual intrinsic volumes \textit{depend} on the dimension of the surrounding space, which means that their values may differ if they are evaluated in subspaces of different dimensions. For this reason, we indicate in our notation $\tW^L_s(K)$ the subspace $L$ in which they are evaluated whenever it differs from $\RR^n$. We finally note that for a convex body $K$ contained in a subspace $L$ of dimension $0<\ell<n$ such that $\cH^\ell(K)>0$ we have
$$
\tW_0^L(K)={\frac 1 \ell}\int\limits_{\SS_L}\varrho_K(u)^\ell\,\cH^{\ell-1}(\dint u)=\cH^\ell(K)=V_\ell(K)\,,
$$
which connects the dual intrinsic volume of order zero and the ordinary intrinsic volumes.

Dual intrinsic volumes turned out to be a crucial and unifying concept for the investigation of intersection bodies and the solution of the Busemann-Petty problem, see \cite{Lutwak88, BusemannPettyGKS} as well as the references cited therein. For a connection between the dual intrinsic volumes and the hyperplane conjecture we refer to \cite{MilmanEmanuelSlicing}.

\paragraph{Poisson hyperplane processes and their zero cell.}
For fixed $r>0$ and $\g>0$, we define the measure $\Theta$ on $A(n,n-1)$ by
\begin{equation}\label{def:Theta}
\Theta(\,\cdot\,) = \frac{\g}{\o_n}\int\limits_{\SS^{n-1}}\int\limits_{\RR} \mathbf{1}\{H(u,t)\in\,\cdot\,\}\,|t|^{r-1}\,\dint t\,\cH^{n-1}(\dint u)\,,
\end{equation}
where
\[H(u,t)=\{x\in\RR^n:\langle x,u\rangle =t\}, \quad u\in \SS^{n-1},t\in \RR.\]
The measure $\Theta$ is rotation invariant for any value of the parameter $r>0$, which is called distance exponent. Moreover, $\Theta$ is translation invariant if and only if $r=1$. We call $\gamma$ the intensity (parameter) associated with $\Theta$ (clearly, $r$ and $\gamma$ are uniquely determined by $\Theta$).

In this paper, we consider Poisson hyperplane processes $X$ in $\mathbb{R}^n$, defined on an underlying probability space $(\Omega,\mathcal{A},\mathbb{P})$, whose intensity measures $\Theta=\mathbb{E} X$ are given by \eqref{def:Theta}, for some distance exponent $r$ and some intensity $\gamma$. Thus, $X$ has the property that for a measurable set $A\subset\SS^{n-1}\times[0,\infty)$ the number of parametrized hyperplanes from $X$ falling in $A$ is Poisson distributed with mean $\Theta(A)$ (usually, one also requires a certain independence property for $X$, which in our situation is automatically fulfilled, see Corollary 3.2.2 in \cite{SW}). The hyperplane process $X$ can be written as $X=\sum_{i\geq 1}\d_{H_i}$, where $\d_{H_i}$ is the unit mass Dirac measure concentrated at $H_i$ and where the random hyperplanes $H_i$, $i\geq 1$, are pairwise distinct. For each realization of $X$, the hyperplanes $H_i$, $i\in\N$,  partition $\RR^n$ into a countable collection of random convex polyhedra, which are called cells in the following. For $H\in X$, let $H^-$ be the closed half-space determined by $H$ which contains the
origin. The random polyhedron
$$Z_0=\bigcap_{H\in X}H^-$$
is the almost surely uniquely determined cell that contains the origin. It is called the zero cell of $X$. The distribution of $Z_0$ is
invariant under rotations and $Z_0$ is almost surely bounded. Hence, $Z_0$ is an isotropic random compact set. Since the hyperplane process
is locally finite, $Z_0$ is indeed almost surely a random polytope (a bounded random polyhedron).

Some special cases of our model are worth to be mentioned. If the distance exponent $r$ equals the space dimension $n$, then $Z_0$ has the same distribution as the typical cell of a translation invariant Poisson-Voronoi tessellation (of suitable intensity), see \cite{HugSchneider}. Moreover, if $r=1$, then $Z_0$ is equal in distribution to the zero cell of a translation and rotation invariant Poisson hyperplane tessellation. These two models are well known and have extensively been studied in the literature, see \cite{SW} and the references cited therein. In this sense, the zero cells $Z_0$  of the Poisson hyperplane processes $X$ from the parametric class we consider interpolate between the typical cell of a Poisson-Voronoi tessellation and the zero cell of a Poisson hyperplane tessellation, see \cite{HoHu}.

\section{Statement of the results}

\subsection{Faces, skeletons and dual intrinsic volumes of the zero cells}\label{subsec:GeneralResults}

A first motivation for our analysis was Efron's identity for the convex hull of random points. The study of convex hulls of uniformly distributed random points placed in a convex domain goes back to the early days of geometric probabilities. In 1864, Sylvester has asked for the probability that the convex hull of four random points in the plane is a triangle. The systematic investigation of random polytopes began with the works \cite{Efron} and \cite{RenyiSulanke} of Efron, and R\'enyi and Sulanke.  For modern developments we refer the interested reader to the survey articles \cite{Barany, HugSurvey,ReitznerSurvey}  and to \cite[Chapter 8.2]{SW}.

Efron's identity connects the number of vertices to the volume of a random polytope. Let $K\subset\RR^n$ be a convex set with unit volume, and let $\eta=\sum_{i\geq 1}\d_{x_i}$ be the restriction to $K$ of a homogeneous Poisson point process in $\RR^n$ of intensity $\l>0$. The convex hull $K_\l={\rm conv}(\eta)$ of the points of $\eta$ is a random polytope contained in $K$. We denote by $\EE f_0(K_\l)$ the mean number of vertices of $K_\l$ and by $\EE V_n(K_\l)$ its mean volume. Since a point $x\in\eta$ is a vertex of $K_\l$ if and only if it is not contained in the convex hull of the other points, we get
$$\EE f_0(K_\l) = \EE\sum_{x\in\eta}{\bf 1}\big\{x\notin{\rm conv}(\eta-\d_x)\big\}\,.$$
Applying Mecke's identity for Poisson point process (see \cite[Theorem 3.2.5]{SW} and also \eqref{eq:Mecke} below), we  conclude that
\begin{equation}\label{eq:Efron}
\begin{split}
\EE f_0(K_\l) 
&=\lambda\,\EE\int\limits_K{\bf 1}\big\{x\notin K_\l\big\}\,\dint x=\l\,\big(1-\EE V_n(K_\l)\big)\,.
\end{split}
\end{equation}
This is the Poissonian analogue of Efron's identity from \cite{Efron} for random convex hulls, which originally deals with a fixed number of random points in $K$.

In the present paper, we explore a setting which is dual to the one described above. First, our random polyhedra are not generated by a collection of random points, but by a random collection of hyperplanes, which form a tessellation of $\RR^n$. Secondly, instead of taking the convex hull we are interested in the zero cell $Z_0$, which arises as an intersection of random half-spaces.

Recently, Schneider \cite{Schneider09} has obtained an Efron-type identity for the zero cell $Z_0$ induced by a translation and rotation invariant Poisson hyperplane process of intensity $\g>0$. For this random polyhedron he proved that
\begin{equation}\label{eq:SchneidersFormel}
\EE f_0(Z_0)=\g^n\,\left(\frac{\k_{n-1}}{ \o_n}\right)^n\,\k_n\,\EE V_n(Z_0)\,.
\end{equation}
In fact, this is a special case of a set of identities obtained in \cite{Schneider09} (see also \eqref{eq:AllgemeineSchneiderformel} below) for stationary but possibly anisotropic hyperplane tessellations and which should be compared with Efron's identity \eqref{eq:Efron}. In analogy to \eqref{eq:Efron}, it relates the combinatorial quantity $\EE f_0(Z_0)$ to the mean volume $\EE V_n(Z_0)$ of $Z_0$.

\medspace

The first part of this paper deals with a generalization of this result for the number of $\ell$-dimensional faces of the zero cell of a Poisson hyperplane tessellation with distance exponent $r >0$ and intensity $\gamma$.
For $0\leq\ell\leq n-1$ we denote by $f_\ell(Z_0)$ the number of $\ell$-dimensional faces of $Z_0$ and by
$${f}(Z_0)=\big(f_0(Z_0),\ldots,f_{n-1}(Z_0)\big)$$ the $f$-vector of $Z_0$. Our first main result relates the combinatorial
quantity ${f}(Z_0)$ to certain metric parameters of $Z_0$. This generalizes the identities from the translation invariant case (\ref{eq:SchneidersFormel}) to our general model. In particular, Theorem \ref{thm:ExpectFvector} includes Efron-type identities for typical cells of Poisson-Voronoi tessellations, as discussed in the introduction.

\begin{theorem}\label{thm:ExpectFvector}
Let $Z_0$ be the zero cell of a Poisson hyperplane process with intensity measure $\Theta$ as in \eqref{def:Theta}, and let $\ell\in\{1,\ldots,n\}$. Then
\begin{equation}\label{eq:MainIdentity}
\EE f_{n-\ell}(Z_0)=c_r(n,\ell)\,\g^\ell\,\EE\tW^{E}_{\ell(1-r)}(Z_0|E)\,,
\end{equation}
where $E={\rm span}\{e_1,\ldots,e_\ell\}$ and where the constant $c_r(n,\ell)$ is given by
\begin{equation*}
\begin{split}
c_r(n,\ell)={\frac {1}{r\, \ell!}}\ & \omega_n^{-\ell}\, \frac {\o_{n-\ell+1}\cdots\o_n}{\o_1\cdots\o_\ell}\\ &\times\,\int\limits_{\left(\SS^{\ell-1}\right)^\ell}\nabla_\ell(u_1,\ldots,u_\ell)^{n-\ell+1}\,\prod_{j=1}^\ell| \langle u_j,e_l\rangle|^{r-1}\,\cH^{\ell(\ell-1)}\big(\dint(u_1,\ldots,u_\ell)\big)\,.
\end{split}
\end{equation*}
\end{theorem}

\medskip

\begin{remark}\rm
Although the factor $\gamma^\ell$ appears in \eqref{eq:MainIdentity}, $\EE f_{n-\ell}(Z_0)$ is independent of the scaling parameter $\gamma$. This is due to the fact that on the right-hand side of \eqref{eq:MainIdentity}, the size of $Z_0$ depends on $\gamma$ and $\tW_{\ell(1-r)}^{E}(Z_0|E)$ measures a metric quantity of $Z_0$.
\end{remark}

It is worth considering the case $\ell=n$ separately. It corresponds to the number of vertices of $Z_0$, where identity \eqref{eq:MainIdentity} in Theorem \ref{thm:ExpectFvector} takes a particularly appealing form.

\begin{corollary}\label{VerticeNumberFormula}
The mean number of vertices of $Z_0$ is given by
$$\EE f_0(Z_0)=\frac {\k_n}{ r}\left(\frac{r}{2}\, \frac{\o_{r+1}}{ \o_{r+n}}\right)^n\,c_r(n)$$
with
$$c_r(n)=\int\limits_{\left(\SS^{n-1}\right)^n}\nabla_n(u_1,\ldots,u_n)\,\prod_{j=1}^n
|\langle u_j,e_n\rangle|^{r-1}\,\cH^{n(n-1)}\big(\dint(u_1,\ldots,u_n)\big)\,.$$
\end{corollary}

For $\ell\in\{0,\ldots,n-1\}$, we are in general not able to simplify the expression obtained for $\EE f_{n-\ell}(Z_0)$ further. However,
since $Z_0$ is almost surely a simple polytope (recall that a polytope is simple if and only if each vertex has exactly $n$ outgoing edges, or if and only if each vertex is contained in exactly $n$ facets),
 with probability one we have $f_1(Z_0)={\frac n 2}f_0(Z_0)$. In particular, for $n=3$ we have the additional Euler relation $f_2(Z_0)=2-f_0(Z_0)+f_1(Z_0)$, so that
$$\EE{ f}(Z_0)=\Big(\EE f_0(Z_0),\,{\frac 3 2}\EE f_0(Z_0),\,2+{\frac 1  2}\EE f_0(Z_0)\Big)\,.$$
This means that for $n=3$ the mean $f$-vector of $Z_0$ is completely determined by the mean number of vertices. For general space dimensions we have the following inequalities.

\begin{corollary}\label{cor:EstimateFVector}
If $\ell\in\{1,\ldots,n-1\}$, then
$$c_r(n,\ell)\,(\ell-1)!\,\o_\ell\,\left(r\,\frac{\o_n\o_{r+1}}{ 2\o_{r+n}}\right)^\ell\leq\EE f_{n-\ell}(Z_0)\leq  \binom{n}{\ell}\,\EE f_0(Z_0)$$
with $c_r(n,\ell)$ as in Theorem \ref{thm:ExpectFvector}.
\end{corollary}

So far we have obtained exact formulae (or upper and lower bounds) for the mean number of faces of $Z_0$. These expressions involve the constants $c_r(n,\ell)$ and $c_r(n)$, which are difficult to handle. For this reason, it is desirable to have upper and lower bounds for $c_r(n,\ell)$, and thus also for $c_r(n)$. Our next result provides such bounds.

\begin{proposition}\label{Prop:ConstantEstimate}
For $\ell\in\{1,\ldots,n\}$ and $r>0$, the inequalities
$$
A(n,\ell,r)
\leq c_r(n,\ell) \leq \ell^{n-\ell+1}A(n,\ell,r)
$$
are satisfied with
$$
A(n,\ell,r)={\frac {2^{\ell}}{r\,\ell!}}\, \frac{\o_r\o_{n-\ell+1}}{\o_\ell\o_{n-\ell+r+1}}\left(\frac{\o_{n+r}}{\o_r\o_n}\right)^\ell\,.
$$
In particular,
$$c_r(n,1)=A(n,1,r)= r^{-1}.$$
\end{proposition}

\begin{corollary}\label{cor:EndabschaetzungFVektor}
If $\ell\in\{1,\ldots,n\}$, then
\begin{equation}\label{fnml}
\frac{\kappa_r}{\ell}\frac{\omega_{n-\ell+1}}{\omega_{n-\ell+1+r}}\left(\frac{\omega_{r+1}}{\kappa_r}\right)^\ell
\leq\EE f_{n-\ell}(Z_0)\le 2\binom{n}{\ell}\left(\frac{\o_{r+1}}{\kappa_r}\right)^{n-1}\,.
\end{equation}
In particular,
\begin{equation}\label{fnull}
{\frac 2 n}\left(\frac {\o_{r+1}}{ \k_r}\right)^{n-1}\leq\EE f_0(Z_0)\leq 2 \left(\frac{\o_{r+1}}{\k_r}\right)^{n-1}\,.
\end{equation}
\end{corollary}

 After having investigated the $f$-vector of $Z_0$, let us finally turn to certain metric parameters of the zero cell. For $\ell\in\{0,\ldots,n-1\}$ the $\ell$-skeleton $\text{skel}_\ell( Z_0) $ of $Z_0$ is the union of all $\ell$-dimensional faces of $Z_0$. Our next result provides an explicit expression for $\EE\mathcal{H}^\ell(\text{skel}_\ell( Z_0))$, the expected $\ell$-dimensional Hausdorff measure of $\text{skel}_\ell( Z_0)$. This generalizes Theorem 10 in \cite{Wieacker} (see also \cite[Equation (10.51)]{SW}) from $r=1$ to general distance exponents.

\begin{theorem}\label{thm:Skelett}
Let $\ell \in \{1,\ldots,n\}.$ Then
\begin{align*}
\mathbb{E}\mathcal{H}^{n-\ell} (\text{skel}_{n-\ell}( Z_0)) &= 2\,c_r(n,\ell) \, \gamma^{-\frac{n-\ell}{ r}} \,
\frac{\o_\ell \o_{\ell r+n-\ell}}{ \o_{\ell r} \o_{2(\ell+\frac{n-\ell }{ r})}}
\left(r\,{\frac \pi 2}  \frac{\o_n\o_{r+1}}{ \o_{r+n}}\right)^{\ell+\frac{n-\ell}{ r}}
\end{align*}
with $c_r(n,\ell)$ as in Theorem \ref{thm:ExpectFvector}.
\end{theorem}

\subsection{Special formulae for $r=1$}\label{subsec:SpecialCaser=1}
Now we turn to Poisson hyperplane processes with an intensity measure as in \eqref{def:Theta} and specialize to the case $r=1$, where $X$  is not only rotation but also translation invariant. The mean $f$-vector of the corresponding zero cell $Z_0$ has been studied for a long time and has turned out to be a notoriously difficult object. The only known explicit result is $\EE f_0(Z_0)=n!\,2^{-n}\,\kappa_n^2$, see \cite[Theorem 10.4.9]{SW}. In \cite{Schneider09} the $f$-vector of $Z_0$ has been studied for translation invariant but anisotropic hyperplane processes. One of the main results of that paper relates $\EE f_{n-\ell}(Z_0)$ to an integral average of expected projection volumes of the zero cell (see \cite[Equation (21)]{Schneider09}). Hence, in the isotropic case, $\EE f_{n-\ell}(Z_0)$ is proportional to the expected $\ell$-th intrinsic volume of $Z_0$ (see \cite[p.~693]{Schneider09}). The latter result is also recovered by our Theorem \ref{thm:ExpectFvector}. We want to go one step further by unifying and extending several formulas for the Poisson zero cell that are available in the literature. For this reason, we define $$F_{\ell;j}(P)=\sum_{F\in\cF_{\ell}(P)}V_j(F)$$
for a polytope $P$ in $\RR^n$, $\ell\in\{0,\ldots,n\}$ and $j\in\{0,\ldots,\ell\}$.

\begin{theorem}\label{thm:VereinheitnlichungSchneiderWieacker}
If $r=1$, then
$$\EE F_{n-\ell;j}(Z_0)=
\g^\ell\,\binom{\ell+j}{ \ell}\,\left(\frac{\kappa_{n-1}}{ \o_n}\right)^\ell\, \frac {\kappa_{\ell+j}}{\kappa_j}\,\EE V_{\ell+j}(Z_0)$$
for all $\ell\in\{0,\ldots,n\}$ and $j\in\{0,\ldots,n-\ell\}$.
\end{theorem}

Let us have a closer look at two particular instances of this identity. If $j=0$, then the formula in Theorem \ref{thm:VereinheitnlichungSchneiderWieacker} reduces to
\begin{equation}\label{eq:AllgemeineSchneiderformel}
\EE F_{n-\ell;0}(Z_0)=\EE f_{n-\ell}(Z_0)=\g^\ell\,\left(\frac{\k_{n-1}}{ \o_n}\right)^\ell\,\k_\ell\,\EE V_\ell(Z_0)\,,
\end{equation}
which is the previously mentioned consequence in \cite{Schneider09}. On the other hand, if $j=n-\ell$, then
\begin{equation*}\label{eq:Wieackerformel}
\EE F_{n-\ell;n-\ell}=
\EE\cH^{n-\ell}({\rm skel}_{n-\ell}(Z_0))=
\g^\ell\, \binom{n}{\ell}\,\left(\frac{\k_{n-1}}{ \o_n}\right)^\ell\,\frac{\k_n}{\k_{n-\ell}}\,\EE V_n(Z_0)\,.
\end{equation*}
This is known from \cite{Wieacker}, see also \cite[Theorem 10.4.9]{SW}.

Furthermore, we can  relate the $f$-vector of $Z_0$ to the quantities $F_{n-\ell;j}$, leading to identities that have, to the best of our knowledge, not been noticed in the literature, except for the case $\ell=0$.

\begin{corollary}\label{cor:VergleichfF}
If $\ell\in\{0,\ldots,n\}$ and $j\in\{0,\ldots,n-\ell\}$, then
$$\EE f_{n-\ell-j}(Z_0)=\g^j\,{\binom {\ell+j}{ \ell}}^{-1}\,\left(\frac{\kappa_{n-1}}{ \o_n}\right)^j\,\k_j\,\EE F_{n-\ell;j}(Z_0)\,.$$
\end{corollary}

Theorem \ref{thm:VereinheitnlichungSchneiderWieacker}  allows us to consider lower-dimensional weighted faces of the tessellation, which are the natural lower-dimensional analogues of the zero cell $Z_0$. For $m\in\{1,\ldots,n-1\}$ we denote by $Z_m$ the $m$-volume weighted typical $m$-face of the rotation and translation invariant tessellation induced by $X$, cf.\ \cite{Schneider09} and \cite{SW} for precise definitions. If $h\geq 0$ is a measurable function on the space of polytopes, a special case of Theorem 1 in \cite{Schneider09} implies that
\begin{equation}\label{eq:lSeiteSchnittdarstellung}
\EE h(Z_m) = \int\limits_{G(n,m)}\EE h(Z_0\cap L)\,\nu_m(\dint L)\,.
\end{equation}
Identity \eqref{eq:lSeiteSchnittdarstellung} can also be used as a definition of (the distribution of) $Z_m$. Similar to the case of the zero cell considered above, we now relate $\EE F_{m-j;i}(Z_m)$ to other parameters of $Z_m$.

\begin{corollary}\label{cor:NiederdimensionaleSeitenF}
If $m\in\{1,\ldots,n\}$, $j\in\{0,\ldots,m\}$ and $i\in\{0,\ldots,m-j\}$, then
\begin{equation*}
\begin{split}
\EE F_{m-j;i}(Z_m) &= \binom{i+j}{ j}\,\left(\frac{\kappa_{n-1}}{\o_n}\right)^j\,\frac{\kappa_{i+j}}{\kappa_i}\,\g^j\,\EE V_{i+j}(Z_m).
\end{split}
\end{equation*}
\end{corollary}

In contrast to the general case, for $r=1$ one of our identities extends to higher-order moments. Namely, we are able to deduce a general relation between the number of facets $f_{m-1}(Z_m)$ and the first intrinsic volume $V_1(Z_m)$ of the $m$-volume weighted typical $m$-face $Z_m$. It seems to be a challenging task to extend this to other functionals.

\begin{theorem}\label{thm:HoehereMomente}
If $m\in\{1,\ldots,n\}$ and $k\in\N$, then
$$\EE V_1^k(Z_m)=\left(\frac{1}{{2\gamma}}\,\frac{\o_n}{ \kappa_{n-1}}\right)^k\,\sum_{q=1}^k\begin{bmatrix}k\\ q\end{bmatrix}\,\EE f_{m-1}^q(Z_m)\,,$$
where {\small $\begin{bmatrix}k\\ q\end{bmatrix}$} are the Stirling numbers of the first kind.
\end{theorem}

\subsection{Sections with subspaces}\label{subsec:PlanarSections}

In the previous subsection, we have seen that for $r=1$ functionals of the $m$-volume weighted typical $m$-face $Z_m$ can be calculated as rotational means of sections of the zero cell with linear subspaces, recall \eqref{eq:lSeiteSchnittdarstellung}. Let us say that $Z_m$ has direction $L\in G(n,m)$ if $Z_m$ is contained in an $m$-dimensional affine subspace parallel to $L$. In \cite{HugSchneider11} the notion of the $m$-volume weighted typical $m$-face $Z_m^{(L)}$ with given direction $L\in G(n,m)$ has been introduced. One can think of the distribution of $Z_m^{(L)}$ as the conditional distribution of $Z_m$, given the direction of $Z_m$ is $L$ (one has to be careful with this interpretation, because the latter event has probability zero). Having this in mind, one can rephrase relation \eqref{eq:lSeiteSchnittdarstellung} by saying that $Z_0\cap L$ has the same distribution as $Z_m^{(L)}$ for $\nu_m$-almost all $L\in G(n,m)$.

For $r\neq 1$, there is no meaningful notion of lower-dimensional typical faces, because of the lack of translation invariance. However, in view of the discussion for $r=1$ above, we consider for general $r>0$ sections of the zero cell $Z_0$ with an $m$-dimensional subspace $L\in G(n,m)$ as analogues of the $m$-volume weighted typical $m$-face $Z_m^{(L)}$ with given direction $L$. On the other hand, one can think of $Z_m^{(L)}$ as the zero cell of the sectional tessellation $X\cap L$, which is defined as
$$X\cap L=\sum_{H\in X,H\cap L\neq\emptyset}\delta_{H\cap L}\,.$$
Since $X\cap L$ is a Poisson process of $(m-1)$-dimensional hyperplanes within $L$, it is characterized by its intensity measure $\Theta_{ L}$. The next result expresses $\Theta_{ L}$ in terms of the underlying hyperplane process $X$ and $L$.

\begin{proposition}\label{prop:IntensityMeasureSectionWithPlane}
Let $m\in\{1,\ldots,n-1\}$ and $L \in G(n,m)$. Then,
\[
\Theta_{ L}(\,\cdot\,)=
\frac{\gamma_m }{\o_m }\int\limits_{\SS_L}\int\limits_{\RR} \mathbf{1}\big\{t\, u + (u^\bot \cap L)\in\,\cdot\,\big\}\, |t|^{r-1} \, \dint t\, \cH^{m-1}(\dint u)
\]
with $\gamma_m=\frac{\o_m\o_{n+r}}{ \omega_n\omega_{m+r}}\,\gamma$.
\end{proposition}

Proposition \ref{prop:IntensityMeasureSectionWithPlane} allows us to translate all results derived in Section \ref{subsec:GeneralResults} for the zero cell $Z_0$ to intersections of $Z_0$ with a fixed subspace $L\in G(n,m)$, since $Z_0\cap L$ is the zero cell of the sectional tessellation $X\cap L$. The only difference is that the parameter $\gamma$ has to be replaced by $\gamma_m$. For example,  Theorem \ref{thm:ExpectFvector} implies that for all $m\in\{1,\ldots,n-1\}$, $\ell\in\{1,\ldots,m\}$ and $L\in G(n,m)$,
$$\EE f_{m-\ell}(Z_0\cap L)=c_r(m,\ell)\,\gamma_m^\ell\,\EE\tW_{\ell(1-r)}^E(Z_0\cap L|E)\,,$$
where $E={\rm span}\{e_1,\ldots,e_\ell\}$ and $c_r(m,\ell)$ is the constant from Theorem \ref{thm:ExpectFvector}. Moreover,
$$
\EE f_0(Z_0\cap L)={\frac{\k_m}{ r}}\left(\frac{r}{2}\,\frac{\o_{r+1}}{ \omega_{r+m}}\right)^m\,c_r(m)
$$
with $c_r(m)$ as in Corollary \ref{VerticeNumberFormula}. For the constant $c_1(m)$ we get
\begin{equation*}
c_1(m) = m!\,\kappa_{m-1}^m\,\kappa_m
\end{equation*}
as a consequence of \cite[Theorem 8.2.2]{SW}. Thus, also $\EE f_0(Z_m)=m!\,2^{-m}\,\kappa_m^2$, a value known from \cite[Theorem 2]{Schneider10}. Similarly, the result of Theorem \ref{thm:Skelett} transfers to the $\ell$-skeleton of $Z_0\cap L$.

\medspace

To prepare for the results in Section \ref{subsec:asymptotikundhyperplane}, we need bounds for higher moments of the volume of $Z_0\cap L$. The following is a re-formulation of Proposition 1 in \cite{HoHu} with $Z_0$ replaced by $Z_0\cap L$ and $\gamma$ by $\gamma_m$.

\begin{proposition}\label{prop:BoundsMomentsVolmeSections}
Let $m\in\{1,\ldots,n-1\}$, $k\in\NN$ and $L \in G(n,m)$. Then,
$$\Gamma\left(\frac{m}{r}+1\right)^k \kappa_m^k\left(r\,\frac{\o_n \o_{r+1}}{2\gamma \o_{n+r}}\right)^{\frac{km}{r}}
\leq\mathbb{E}V_m^k(Z_0 \cap L)
\leq\Gamma\left(\frac{km}{r}+1\right)\kappa_m^k\left(r\,\frac{\o_n \o_{r+1}}{2\gamma \o_{n+r}}\right)^{\frac{km}{r}}\,.$$
In particular, for $k=1$ we have the exact formula
$$\mathbb{E}V_m(Z_0 \cap L)=\Gamma\left(\frac{m}{r}+1\right)\kappa_m\left(r\,\frac{\o_n \o_{r+1}}{2\gamma \o_{n+r}}\right)^{\frac{m}{r}}.$$
\end{proposition}

\begin{remark}\rm Proposition \ref{prop:BoundsMomentsVolmeSections} continues to hold for $m=n$, in which case $Z_0\cap L$ has to be interpreted as $Z_0$ itself. In particular, the mean volume of $Z_0$ equals
\begin{equation}\label{eq:MeanVolumeZ0}
\mathbb{E}V_n(Z_0)=\Gamma\left(\frac{n}{r}+1\right)\kappa_n\left(r\,\frac{\o_n \o_{r+1}}{2\gamma \o_{n+r}}\right)^{\frac{n}{r}},
\end{equation}
see Proposition 1 in \cite{HoHu}.
\end{remark}

To rephrase bounds for the variance $\var(V_m(Z_0 \cap L))$ of $V_m(Z_0\cap L)$, we need auxiliary notation and results from \cite[Section 4]{HoHu}. Firstly, we need the quantity $E(m,r)$ which is defined for $m\in \NN$ and $r>0$ as an involved multiple integral, see \cite{HoHu}. Instead of stating its definition we recall upper and lower bounds derived in \cite[Lemma 5]{HoHu}, which are the only information about $E(m,r)$ needed later. Namely, if $r>0$ and $m\geq 3$, then
\begin{align}
&c\,\frac{\left(1+\frac{r}{m}\right)^{-\frac{1}{2}}}{\sqrt{r+1}}\,2^{\frac{m}{2}}\,\left(1+\frac{r}{2m}\right)^{-\frac{m}{2}}
\left(1+\frac{m}{m+r}\right)^{-\frac{m+r}{2}}\nonumber\\
&\qquad\qquad \leq
E(m,r)\leq C \,\frac{(1+\frac{r}{m})}{\sqrt{r+1}} \,2^{\frac{m}{2}}\, \left(1+\frac{r}{2m}\right)^{-\frac{m}{2}}\left(1+\frac{m}{m+r}\right)^{-\frac{m+r}{2}}\label{BoundForE}
\end{align}
with constants $c, C>0$ which  are independent of $r$ and $m$.
For integers $0<m<n$ we also introduce the quantity $D(n,m,r)$ by
\begin{equation}\label{AuxiliaryConstantD(n,m,r)}
D(n,m,r) = \frac{m\kappa_m^2}{r}\, \Gamma\left(\frac{2m}{r}+1\right)\,\left(\frac{r}{4\gamma}\frac{ \o_n\o_{r+1}}{\o_{n+r} }\right)^{\frac{2m}{r}}.
\end{equation}
It corresponds to $D(m,r)$ in \cite[Section 4]{HoHu} with $\gamma$ replaced by $\gamma_m$ there. Now, the following bounds are direct consequences of Theorem 2 in \cite{HoHu}.

\begin{proposition}\label{prop:BoundsVarianceSections}
For $m\in\{1,\ldots,n-1\}$ and $L \in G(n,m)$ we have
$$D(n,m,r)\, E(m,r)\leq\var(V_m(Z_0 \cap L))\leq 4^{\frac{2m}{r}+1}\,D(n,m,r)\, E(m,r)\,. $$
\end{proposition}

\subsection{Relation to the hyperplane conjecture}\label{subsec:asymptotikundhyperplane}

Recall that the hyperplane conjecture asserts the existence of a universal constant $c > 0$ such that for any  convex body $K\subset\R^n$ of volume one there is a hyperplane $L \in A(n,n-1)$ with
$$V_{n-1}(K \cap L)> c\,.$$
In this section, we show that the hyperplane conjecture holds asymptotically almost surely for the suitably normalized zero cells generated by the hyperplane process $X$ with distance exponent $r=b\,n^\alpha$ where $b>0$ and $\alpha>1/2$ or $b>\sqrt{8}$ and $\a=1/2$. This implies in particular that the hyperplane conjecture holds asymptotically almost surely for the typical cell of a translation invariant Poisson-Voronoi tessellation. We start with the following consequence of Propositions \ref{prop:BoundsMomentsVolmeSections} and \ref{prop:BoundsVarianceSections}.

\begin{corollary}\label{cor:AsymptotikVolumeSection}
Let $r=b\, n^\alpha$ with $b>0$ and $\alpha\in\RR$. Let $Z_0$ be the zero cell of a Poisson hyperplane process with
distance exponent $r$ and intensity $\widehat{\gamma}(r,n)$ given by
\begin{equation}\label{eq:GammaHatAN}
\widehat{\gamma}(r,n) = \frac{r}{2} \frac{\o_n\o_{r+1}}{\o_{n+r}} \left({ \Gamma\left(\frac{n}{r}+1\right) \kappa_n  } \right)^{r/n}\,.
\end{equation}
Then $\EE V_n(Z_0)=1$ for all choices of $r$ and $n\geq 2$. Let $\ell\in\N_0$ be fixed and let $L \in G(n,n-\ell)$. Then as $n\to \infty$ we obtain that
\[
\lim\limits_{n\rightarrow \infty} \mathbb{E}V_{n-\ell}(Z_0\cap L) =\begin{cases}
 0&:\alpha<0 \text{ and }\ell >0 \\
e^{-\frac{\ell}{b}+\frac{\ell}{2}}&:\alpha=0\\
e^{\ell/2}&:\alpha>0\\
 \end{cases}
\]
and
\begin{align*}
\var(V_{n-\ell}(Z_0\cap L))\rightarrow\begin{cases}
\infty &:\alpha<0 \text{ and }\ell=0\\
\text{open} &:\alpha <0 \text{ and }\ell>0\\
\infty&:\alpha=0\\
\text{open} &:\alpha\in (0,1/2)\text{ or }\\
 &\;\;\alpha= 1/2 \text{ and } b\in (0,\sqrt{8}\,]\\
 0 &:\alpha= 1/2 \text{ and } b> \sqrt{8}\text{ or }\\
 &\;\;\alpha>1/2,
\end{cases}
\end{align*}
where `open' means that the behaviour cannot be deduced from Proposition \ref{prop:BoundsVarianceSections} since the lower bound converges to zero and the upper bound to $\infty$. Furthermore, the rate of convergence for $\a=1/2$ and $b>\sqrt{8}$ or $\alpha>1/2$ is given as follows:
\begin{itemize}
\item[{\rm (i)}]
If $\alpha=1/2$ and $b>\sqrt{8}$, then
\[
(2^{-b/2}-\epsilon)^{\sqrt{n}}\leq \var(V_{n-\ell}(Z_0\cap L))\leq (2^{4/b-b/2}+\epsilon)^{\sqrt{n}}\,,\]
for every $\epsilon\in (0,2^{-b/2})$ and all $n\geq N_{\e}$, for some $N_{\e}\in\NN$.
\item[{\rm (ii)}]
If $1/2<\alpha<1$, then
\[
(2^{-b/2}-\epsilon)^{n^\alpha}\leq \var(V_{n-\ell}(Z_0\cap L))\leq (2^{-b/2}+\epsilon)^{n^\alpha}\,,\]
for every $\epsilon\in (0,2^{-b/2})$ and all $n\geq N_{\e}$, for some $N_{\e}\in\NN$.
\item[{\rm (iii)}]
If $\alpha=1$, then
\[
c_1\frac{1}{\sqrt{n}}\left( \frac{4(b+1)^{b+1}}{(b+2)^{b+2}}\right)^{n/2}\leq\var(V_{n-\ell}(Z_0\cap L))\leq c_2\frac{1}{\sqrt{n}}\left( \frac{4(b+1)^{b+1}}{(b+2)^{b+2}}\right)^{n/2}
\]
with constants $c_1,c_2>0$ not depending on $n$.
\item[{\rm (iv)}]
If $\alpha>1$, then
\[
\left(c_1 n^{1-\alpha}\right)^{n/2}\leq \var(V_{n-\ell}(Z_0\cap L))\leq \left(c_2 n^{1-\alpha}\right)^{n/2}
\]
with constants $c_1,c_2>0$ not depending on $n$.
\end{itemize}
\end{corollary}
Since the hyperplane conjecture refers to convex bodies with volume one, we define the normalized zero cell $\overline{Z}_0$ by $\overline{Z}_0 = (V_n(Z_0))^{-1/n}\,Z_0$.
Hence, for $L \in G(n,n-1)$, we have
\[
V_{n-1}(\overline{Z}_0\cap L)
 =(V_n(Z_0))^{-(n-1)/n}\,V_{n-1}(Z_0 \cap L)\,.
\]
We can now present the main result of this section.

\begin{theorem}\label{thm:HyperplaneConjecture}
Let $r = b\,n^\alpha$ with $b>0$ and $\alpha>1/2$ or $b>\sqrt{8}$ and $\a=1/2$, and let the intensity of the underlying hyperplane process $X$ be given as at \eqref{eq:GammaHatAN}. Then, for any $\varepsilon \in (0, \sqrt{e})$ and $L \in G(n,n-1)$ it holds that
\[
\mathbb{P}\big(V_{n-1}(\overline{Z}_0 \cap L) > \sqrt{e}-\varepsilon\big)\, \geq\,
1-  C\,\varepsilon^{-2}\, H(\alpha,b,n),
\]
where
\[
H(\alpha,b,n)=\begin{cases}
\left( 2^{2/b-b/4}\right)^{\sqrt{n}}&: \a=1/2 \text{ and }b>\sqrt{8}\\[0.2cm]
\left(2^{-b/4}\right)^{n^\alpha}&:1/2<\alpha<1\\[0.2cm]
 \frac{1}{\sqrt{n}} \left(\frac{4(b+1)^{b+1}}{(b+2)^{b+2}}\right)^{n/2}&:\alpha=1\vspace{0.2cm}\\
(C\,n^{1-\alpha})^{n/2}&:\alpha>1
\end{cases}
\]
 with $n\geq N_\varepsilon$, for some $N_\varepsilon\in\NN$, and with a constant $C > 0$ not depending on $n$.
 In particular, the hyperplane conjecture holds asymptotically almost surely for the normalized zero cell $\overline{Z}_0$ in the sense that
\[
\lim\limits_{n\rightarrow \infty}\mathbb{P}(V_{n-1}(\overline{Z}_0 \cap L) > \sqrt{e}-\varepsilon) = 1
\]
for all $\varepsilon \in (0, \sqrt{e})$ and $L \in G(n,n-1)$.
\end{theorem}

\begin{remark}\rm
It is interesting to compare the result of the previous theorem with another approach to the hyperplane conjecture. Namely, it is known that if $K$ is a convex body in $\RR^n$ with unit volume, then
\begin{equation}\label{eq:HyperplaneAndLK}
\sup \limits_{L \in A(n,n-1)} V_{n-1}(K \cap L) \geq \frac{C}{L_K}\,,
\end{equation}
where $C>0$ is a universal constant and $L_K$ denotes the isotropic constant of $K$ (see \cite{KK}). Hence, if one could show the existence of a constant $c>0$ with $L_K \leq c$ for every convex body $K$ of unit volume and dimension $n$, the hyperplane conjecture would follow immediately. To connect \eqref{eq:HyperplaneAndLK} with our analysis from Section \ref{subsec:HighDimensions}, let us recall from \cite{AlonsoGutierrez.2010} that for polytopes $P \subset \mathbb{R}^n$ with volume $0<V_n(P)<\infty$ one has the relation
\begin{equation}\label{eq:KKRelation}
L_P \leq C'\,\sqrt{\frac{f_0(P)}{n}}
\end{equation}
for an absolute constant $C'>0$ not depending on $P$ or $n$. For distance exponents $r=b\,n^\a$ with $b>0$ and $\a>0$ and intensities given by \eqref{eq:GammaHatAN}, we can find -- as a consequence of Theorem \ref{thm:HighDimensionsVertices} -- constants $0<c_1,c_2<\infty$ such that
\[
 (c_1\; n^\a)^{n/2} \; \leq \; \mathbb{E}f_0(Z_0)=\mathbb{E}f_0(\overline{Z}_0) \; \leq \;
 (c_2\;n^\a)^{n/2}\,.
\]
This strongly suggests that the result of Theorem \ref{thm:HyperplaneConjecture} cannot be derived by combining \eqref{eq:HyperplaneAndLK} with \eqref{eq:KKRelation}.
\end{remark}

\begin{remark}\rm
So far, we have considered the case $r=b\,n^\alpha$ with $\alpha>1/2$ or $\a=1/2$ and $b>\sqrt{8}$ only. In fact, for $\alpha\leq 0$ our method, which is based on the use of Chebychev's inequality, does not lead to a result similar to that of Theorem \ref{thm:HyperplaneConjecture}. This is due to the fact that for $\alpha\leq 0$ the variance of $V_{n}(Z_0)$ tends to $\infty$ as $n\to\infty$. For $0<\alpha< 1/2$ or $\a=1/2$ and $b\leq \sqrt{8}$ the asymptotic behaviour of the variance is still open. Searching for a counterexample to the hyperplane conjecture, the case $\alpha<0$ seems to be most promising because then for every hyperplane $L\in G(n,n-1)$ the expected sectional volume $\EE V_{n-1}(Z_0\cap L)$ converges to zero.
\end{remark}

\subsection{High dimensional limits}\label{subsec:HighDimensions}

In this section we investigate the behaviour of the isoperimetric ratio and of the $f$-vector of the zero cells $Z_0$, as the space dimension $n$ tends to infinity.  To start with, let us define for a random convex body $K\subset\RR^n$ the isoperimetric ratio $\cI_n(K)$ of half the expected surface area and the expected volume by
$$\cI_n(K)={\frac{(\EE V_{n-1}(K))^{n/(n-1)}}{\EE V_n(K)}}.$$
In this context, we call a function $\varphi:\NN\to\RR$ a gauge function for a sequence of random convex bodies $K_n \subset \RR^n$ if
$$\lim_{n\to\infty}\varphi (n)\,\cI_n(K_n)=1\,.$$
For example, $\varphi(n)=\frac 1 {\sqrt {2\pi e}}\, n^{-1/2} $ defines a gauge function for the Euclidean unit ball, interpreted as the constant random convex body $K\equiv\BB^n$.

\begin{theorem}\label{thm:IsoperimetricRatio}
If the distance exponent is $r=b\, n^\alpha$ for some $b>0$ and $\a\in\RR$, then a gauge function for the zero cell $Z_0$ is
\begin{align*}
\varphi(n)=\begin{cases}
\frac{b}{2\sqrt{e}}\,n^{\a-1} \left(1-\frac{1}{n}\right)^{-n^{1-\a}}&:\a<0 \\[1.5ex]
e^{\frac{1}{b}-\frac{1}{2}}\frac{\o_b}{\o_{b+1}}\,n^{-1}&:\a=0\\[1.5ex]
\sqrt{\frac{b}{2\pi e}}\,n^{-(1-\a/2)}&:0<\a<1\\[1.5ex]
\sqrt{\frac{b}{(b+1)(2e\pi)}}\,n^{-1/2}&:\a=1
\vspace{0.3cm}\\[1ex]
\frac{1}{\sqrt{2\pi e}}\,n^{-1/2}&:\a>1\,.
\end{cases}
\end{align*}

\end{theorem}

This result is independent of the intensity, since the intensities cancel out in the ratio.
The most surprising observation is that only for $\alpha<0$ the gauge function is divergent, as $n$ tends to infinity, since the term $\left(1-\frac{1}{n}\right)^{-n^{1-\a}}$ grows exponentially fast.
This indicates that for negative $\alpha$ the asymptotic nature of the zero cell is fundamentally different.
Theorem \ref{thm:IsoperimetricRatio} in particular covers the case of the zero cell of a rotation and translation invariant Poisson hyperplane tessellation ($r=1$) and that of the typical cell of a translation invariant Poisson-Voronoi tessellation ($r=n$) and highlights the different shapes of these cells in high  dimensions. Comparing the gauge functions with that of a ball, we roughly speaking see that typical Poisson-Voronoi cells are approximately spherical in the mean, whereas the shape of the zero cell for $r=1$ is degenerate, in this sense.

\medspace

Next we investigate the behaviour of the $f$-vector of the zero cell $Z_0$ in different asymptotic regimes. The limiting behaviour of the bounds from \eqref{fnull} in Corollary \ref{cor:EndabschaetzungFVektor} can be obtained by applying Stirling's formula. This yields the growth rate for the expected number of vertices of the zero cell  in the general regime $r=b\,n^\a$ with $b>0$ and $\a\in\RR$. Since $Z_0$ is almost surely a simple polytope, $f_1(Z_0)={\frac{n}{ 2}}f_0(Z_0)$ almost surely, and thus we also obtain a corresponding result for the number of edges. A refinement of this argument, based on \eqref{fnml} in Corollary \ref{cor:EndabschaetzungFVektor}, exhibits the asymptotic behaviour for the other face numbers as well.

\begin{theorem}\label{thm:HighDimensionsVertices}
Let $r=b\,n^\a$ with $b>0$ and $\a\in \RR$. Let $\ell\in\N_0$ be fixed.
\begin{enumerate}
\item[{\rm (i)}] If $\alpha<0$, then
\[\lim\limits_{n\to\infty}\sqrt[n]{\EE f_\ell(Z_0)} =2.\]
\item[{\rm (ii)}] If $\alpha=0$, then
\[\lim\limits_{n\to\infty}\sqrt[n]{\EE f_\ell(Z_0)}=\frac{\o_{b+1}}{\k_b}.\]
\item[{\rm (iii)}] If $\alpha>0$, then
\[\lim\limits_{n\to\infty} n^{-\a/2}\sqrt[n]{\EE f_\ell(Z_0)}=\sqrt{2\pi b}.\]
\end{enumerate}
\end{theorem}

In (i) and (ii) of the following theorem, our approach yields $\liminf\limits_{n\to\infty}\sqrt[n]{\EE f_{n-\ell}(Z_0)}\ge 1$, which is trivial and therefore not stated as part of the result.
\begin{theorem}\label{thm:HighDimensionRfest}
Let $r=b\, n^\alpha$ with $b>0$ and $\alpha\in\R$. Let $\ell \in\N$ be fixed.
\begin{enumerate}
\item[{\rm (i)}] If $\a\leq 0$, then
\[\limsup\limits_{n\to\infty}\sqrt[n]{\EE f_{n-\ell}(Z_0)}\leq \begin{cases}2& : \a<0\\ \frac{\o_{b+1}}{\k_b}& : \a=0.\end{cases}\]
\item[{\rm (ii)}] If $\a\in(0,1)$, then
\[\limsup\limits_{n\to\infty}n^{-\a/2}\sqrt[n]{\EE f_{n-\ell}(Z_0)}\leq\sqrt{2\pi b}.\]
\item[{\rm (iii)}]If $\a=1$, then
\[\sqrt{1+b}\left( 1+\frac{1}{b}\right)^{\frac{b}{2}}\leq \liminf\limits_{n\to\infty}\sqrt[n]{\EE f_{n-\ell}(Z_0)},\quad \limsup\limits_{n\to\infty}n^{-\frac{\alpha}{2}}\sqrt[n]{\EE f_{n-\ell}(Z_0)}\leq \sqrt{2\pi b}.\]
\item[{\rm (iv)}]If $\a>1$, then
\[\sqrt{eb}\leq \liminf\limits_{n\to\infty}n^{\frac{1-\a}{2}}\sqrt[n]{\EE f_{n-\ell}(Z_0)}, \quad \limsup\limits_{n\to\infty}n^{-\frac{\a}{2}}\sqrt[n]{\EE f_{n-\ell}(Z_0)}\leq \sqrt{2\pi b}.\]
\end{enumerate}
\end{theorem}

In the preceding two theorems, we considered the growth rates of $\EE f_{\ell} (Z_0)$ and of $\EE f_{n-\ell} (Z_0)$ for fixed $\ell$
and $n\to\infty$. We complement the picture by studying an intermediate regime, where $\ell$ is proportional to $n$, that is, $\ell=\lfloor an\rfloor$ with $a\in(0,1)$.

\begin{theorem}\label{thm:HighDimensionsR=bn}
Let $r=b\,n^\alpha$ with $b>0$ and $\alpha\in\R$. Let $a\in (0,1)$ and $\ell=\lfloor an\rfloor$. Put $c(a)=(a^a(1-a)^{1-a})^{-1}$.
\begin{enumerate}
\item[{\rm (i)}] If $\alpha<0$, then
$$
2^{1-a}\le \liminf\limits_{n\to\infty}\sqrt[n]{\EE f_{\ell}(Z_0)},\qquad \limsup\limits_{n\to\infty}\sqrt[n]{\EE f_{\ell}(Z_0)}\le 2\, c(a)\,.
$$
\item[{\rm (ii)}] If $\alpha=0$, then
$$
\left(\frac{\omega_{b+1}}{\kappa_b}\right)^{1-a}\le \liminf\limits_{n\to\infty}\sqrt[n]{\EE f_{\ell}(Z_0)},\qquad \limsup\limits_{n\to\infty}\sqrt[n]{\EE f_{\ell}(Z_0)}\le c(a)\, \frac{\omega_{b+1}}{\kappa_b}\,.
$$
\item[{\rm (iii)}]
If $\alpha>0$, then
$$
\limsup\limits_{n\to\infty}n^{-\alpha/2}\sqrt[n]{\EE f_{\ell}(Z_0)}\le c(a)\,\sqrt{2\pi b} \,,
$$
$$
\liminf\limits_{n\to\infty}n^{-\alpha (1-a)/2}\sqrt[n]{\EE f_{\ell}(Z_0)}\ge \begin{cases} \sqrt{2\pi b}^{1-a}& : \alpha\in (0,1)\\
 \sqrt{2\pi b}^{1-a}\left(1+\frac{b}{a}\right)^{\frac{a}{2}}\left(1+\frac{a}{b}\right)^{\frac{b}{2}} & : \alpha=1
   \end{cases}
$$
and
$$
\liminf\limits_{n\to\infty}n^{-(\alpha-a)/2}\sqrt[n]{\EE f_{\ell}(Z_0)}\ge\sqrt{2\pi b}^{1-a}\left(\frac{eb}{a}\right)^{\frac{a}{2}}\,,\qquad  \alpha>1\,.
$$
\end{enumerate}
\end{theorem}

\section{Proofs}\label{sec:Proofs}

\subsection{Preparations}

We will often use the following well-known fact: if $X$ is a non-negative random variable with $\EE X^p<\infty$ for some integer $p\geq 1$, then
\begin{equation}\label{eq:Momentenformel}
\EE X^p = p\,\int\limits_0^\infty s^{p-1}\,\PP(X>s)\,\dint s\,.
\end{equation}

Let us recall the multivariate Mecke formula for the Poisson hyperplane process $X$, see \cite[Corollary 3.2.3]{SW}. It is one of our main tools to establish the identity stated in Theorem \ref{thm:ExpectFvector}. For $\ell\in\N$ and any non-negative measurable function $f$ depending on $\ell$ hyperplanes and on $X$, it states that
\begin{equation}\label{eq:Mecke}
\begin{split}
&\EE\sum_{(H_1,\ldots,H_\ell)\in X_{\neq}^\ell}f(H_1,\ldots,H_\ell,X)\\ &\qquad\qquad\qquad=\int\limits_{A(n,n-1)^\ell}\EE f(H_1,\ldots,H_\ell,X+\d_{H_1}+\ldots+\d_{H_\ell})\,\Theta^\ell\big(\dint(H_1,\ldots,H_\ell)\big)\,,
\end{split}
\end{equation}
where $X+\d_{H_1}+\ldots+\d_{H_\ell}$ is the Poisson hyperplane process $X$ with the hyperplanes $H_1,\ldots,H_\ell$ added and
$X^\ell_{\neq}$ denotes the set of all $(H_1,\ldots,H_\ell)$ with $H_i\in X$ for $i=1,\ldots,n$ and $H_i\neq H_j$ for $i\neq j$.

In the following, we will need an expression for the $\Theta$-measure of the set of hyperplanes that intersect a line segment having one of its endpoints at the origin. Recall that $\Theta$ is given by \eqref{def:Theta}.

\begin{lemma}\label{lem:ThetaStrecke0z}
If $z\in\RR^n$, then
$$\Theta\left(A_{[0,z]}\right)={\frac{2\g}{ r}}{\frac{\o_{r+n}} {\o_n\o_{r+1}}}\,\|z\|^r\,,$$
where $A_{[0,z]}=\{H\in A(n,n-1):H\cap [0,z]\neq\emptyset\}$ is the set of hyperplanes which intersect the line segment $[0,z]=\{sz:s\in[0,1]\}$.
\end{lemma}
\begin{proof}
Using \eqref{def:Theta} we can write
\begin{equation*}
\begin{split}
\Theta\left(A_{[0,z]}\right) &= {\frac{\g}{\o_n}}\int\limits_{\SS^{n-1}}\int\limits_{\RR}{\bf 1}\big\{H(u,t)\cap[0,z]\neq\emptyset\big\}\,|t|^{r-1}\;\dint t\,\cH^{n-1}(\dint u)\\
&= {\frac{2\g}{\o_n}}\int\limits_{\SS^{n-1}}\int\limits_0^\infty{\bf 1}\big\{0\leq t\leq\langle z,u\rangle_+\big\}\,t^{r-1}\;\dint t\,\cH^{n-1}(\dint u)\\
&= {\frac{2\g}{\o_n}}{\frac{\|z\|^r}{ r}}\int\limits_{\SS^{n-1}}\langle e_z,u\rangle_+^r\;\cH^{n-1}(\dint u)\,,
\end{split}
\end{equation*}
where $e_z\in\SS^{n-1}$ is such that $z=\|z\|e_z.$
Expressing $u\in\SS^{n-1}$ as $u=\cos\theta \,e_z+\sin\theta \,v$ with $0\leq\theta\leq\pi$ and $v\in\SS^{n-1}\cap e_z^\perp$ and using \eqref{eqGamma}, we obtain
\begin{equation*}
\begin{split}
\int\limits_{\SS^{n-1}}\langle e_z,u\rangle_+^r\,\cH^{n-1}(\dint u) &= \int\limits_0^\pi\int\limits_{\SS^{n-1}\cap e_z^\perp}\max\{0,\langle e_z,\cos\theta\, e_z+\sin\theta\, v\rangle\}^r\,\cH^{n-2}(\dint v)\,(\sin\theta)^{n-2}\;\dint\theta\\
&=\omega_{n-1}\int\limits_0^{\pi/2}(\cos\theta)^r(\sin\theta)^{n-2}\;\dint\theta ={\frac{\o_{r+n}}{\o_{r+1}}}\,.
\end{split}
\end{equation*}
Thus,
$$\Theta(A_{[0,z]})={\frac{2\g}{ r}}{\frac{\o_{r+n}} {\o_n\o_{r+1}}}\,\|z\|^r\,,$$
which completes the proof.
\end{proof}

\subsection{Proof of Theorem \ref{thm:ExpectFvector} and its corollaries}\label{sec:proofExactFormula}

\begin{proof}[Proof of Theorem \ref{thm:ExpectFvector}.]
For $\ell\in\{0,\ldots,n\}$, any $(n-\ell)$-dimensional face $F\in\cF_{n-\ell}(Z_0)$ of $Z_0$ is the intersection $Z_0\cap H_1\cap\ldots\cap H_\ell$ of the zero cell with $\ell$ hyperplanes from $X$. On the other hand, it follows as in the proof of \cite[Theorem  4.4.5]{SW} that almost surely any $\ell$ distinct hyperplanes from $X$ have linearly independent normal vectors.
Therefore, almost surely every nonempty intersection $Z_0\cap H_1\cap\ldots\cap H_\ell$ of the zero cell with $\ell$ hyperplanes from $X$ is an $(n-\ell)$-dimensional face of $Z_0.$  Thus, we can re-write $\EE f_{n-\ell}(Z_0)$ as
$$\EE f_{n-\ell}(Z_0)={\frac{1}{\ell!}}\,\EE\sum_{(H_1,\ldots,H_\ell)\in X_{\neq}^\ell}{\bf 1}\{H_1\cap\ldots\cap H_\ell\cap Z_0\neq\emptyset\}$$
and use the multivariate Mecke formula \eqref{eq:Mecke} together with \eqref{def:Theta} to see that
\begin{equation*}
\begin{split}
\EE f_{n-\ell}(Z_0)
&= {\frac{1}{\ell!}}\left({\frac{\g}{\o_n}}\right)^\ell \EE\int\limits_{(\SS^{n-1})^\ell}\int\limits_{\RR^{\ell}}{\bf 1}\{H(u_1,t_1)\cap\ldots\cap H(u_\ell, t_\ell)\cap Z_0\neq\emptyset\}\\
&\qquad\qquad\qquad\times\,|t_1\cdots t_\ell|^{r-1}\,\dint( t_1,\ldots, t_\ell)\,\cH^{\ell(n-1)}\big(\dint(u_1,\ldots,u_\ell)\big).
\end{split}
\end{equation*}
Fix linearly independent $u_1,...,u_\ell\in \SS^{n-1}$ and put $U={\rm span}\{u_1,\ldots,u_\ell\}$. For $t_1,\ldots,t_\ell\in\RR$ let $T(t_1,\ldots,t_\ell)$ be the intersection point of the hyperplanes $H(u_1,t_1),\ldots,H(u_\ell,t_\ell)$ with $U$. The mapping $T:\RR^{\ell}\to U$ is bijective and its inverse is given by $T^{-1}(z)=(\langle z,u_1\rangle,\ldots,\langle z,u_\ell\rangle)$. The Jacobian of $T^{-1}$ is $\nabla_\ell(u_1,...,u_\ell)$, that is, the $\ell$-dimensional volume of the parallelepiped spanned by $u_1,\ldots,u_\ell$, see \cite[Equation (13)]{Schneider09}. Moreover, we have
$$H(u_1,t_1)\cap\ldots\cap H(u_\ell, t_\ell)\cap Z_0\neq\emptyset$$
if and only if $\left(T(t_1,\ldots,t_\ell)+U^\bot\right)\cap Z_0\neq \emptyset$ which is equivalent to $T(t_1,\ldots,t_\ell)\in Z_0|U$.
 Thus,
\begin{equation*}
\begin{split}
\EE f_{n-\ell}(Z_0)={\frac{1}{\ell!}}\left({\frac{\g}{\o_n}}\right)^\ell \EE\int\limits_{(\SS^{n-1})^\ell}&\int\limits_U\nabla_\ell(u_1,\ldots,u_\ell)
\,{\bf 1}\{z\in Z_0|U\}\\
&\times\prod_{j=1}^\ell|\langle u_j,z\rangle|^{r-1}\,\cH^\ell(\dint z)\,\cH^{\ell(n-1)}\big(\dint(u_1,\ldots,u_\ell)\big).
\end{split}
\end{equation*}
Since the integrand is symmetric in $u_1,\ldots,u_\ell\in \SS^{n-1}$ we can replace the integration over $\left(\SS^{n-1}\right)^{\ell}$ by an integration over $\left(G(n,1)\right)^{\ell}$. Then, we apply in a second step an integral-geometric transformation formula of Blaschke-Petkantschin-type, \cite[Theorem 7.2.3]{SW} (alternatively, the more general Theorem 1 in \cite{ArbeiterZaehle} can be applied directly), which implies that the integration over $\left(G(n,1)\right)^{\ell}$ is replaced by an integration over $G(n,\ell)$. Thus, we obtain
\begin{equation*}
\begin{split}
\EE f_{n-\ell}(Z_0) &= {\frac{b}{\ell!}}\left({\frac{\g}{\o_n}}\right)^\ell\EE\int\limits_{G(n,\ell)} \int\limits_{(\SS_L)^\ell}\int\limits_L\nabla_\ell(u_1,\ldots,u_\ell)^{n-\ell+1} \,{\bf 1}\{z\in Z_0|U\}\\
&\hspace{3cm}\times\prod_{j=1}^\ell|\langle u_j,z\rangle|^{r-1}\,\cH^\ell(\dint z)\,\cH^{\ell(\ell-1)}\big(\dint(u_1,\ldots,u_\ell)\big)\,\nu_\ell(\dint L)\\
&= {\frac{b}{\ell!}}\left({\frac{\g}{\o_n}}\right)^\ell\EE\int\limits_{G(n,\ell)}
\int\limits_{(\SS_L)^\ell}\int\limits_L\nabla_\ell(u_1,\ldots,u_\ell)^{n-\ell+1}\,\|z\|^{\ell(r-1)}\,{\bf 1}\{z\in Z_0|L\}\\
&\hspace{3cm}\times\prod_{j=1}^\ell|\langle u_j,e_z\rangle|^{r-1}\,\cH^\ell(\dint z)\,\cH^{\ell(\ell-1)}\big(\dint(u_1,\ldots,u_\ell)\big)\,\nu_\ell(\dint L),
\end{split}
\end{equation*}
where $e_z$ is a unit vector such that $z=\|z\|\,e_z$. The precise value of the constant $b$ follows from Equation (7.12) in \cite{SW} and equals
\begin{equation}\label{eq:DefConstantB}
 b={\frac{\o_{n-\ell+1}\cdots\o_n}{\o_1\cdots\o_\ell}}\,.
\end{equation}
Let $E=\text{span}\,\{e_1,\ldots,e_\ell\}$ and $\r\in SO_n $ be such that $\r L=E$ and $\r e_z = e_\ell$. Then
\begin{equation*}
\begin{split}
&\int\limits_{(\SS_L)^\ell}\nabla_\ell(u_1,\ldots,u_\ell)^{n-\ell+1}\,\prod_{j=1}^\ell
|\langle u_j,e_z\rangle|^{r-1}\,\cH^{\ell(\ell-1)}\big(\dint(u_1,\ldots,u_\ell)\big)\\
&=\int\limits_{(\SS_L)^\ell}\nabla_\ell(\r u_1,\ldots,\r u_\ell)^{n-\ell+1}\,\prod_{j=1}^\ell|\langle \r u_j,\r e_z\rangle|^{r-1}\,\cH^{\ell(\ell-1)}\big(\dint(u_1,\ldots,u_\ell)\big)\\
&=\int\limits_{(\SS_E)^\ell}\nabla_\ell( u_1,\ldots, u_\ell)^{n-\ell+1}\,\prod_{j=1}^\ell|\langle u_j,e_\ell \rangle|^{r-1}\,\cH^{\ell(\ell-1)}\big(\dint(u_1,\ldots,u_\ell)\big)=:c,
\end{split}
\end{equation*}
which is independent of $z$ and $L$. Using the isotropy of $Z_0$, we obtain
\begin{align*}
\EE f_{n-\ell}(Z_0) &= {\frac{bc}{\ell!}}\left({\frac{\g}{\o_n}}\right)^\ell\EE\int\limits_{G(n,\ell)}\int\limits_L\|z\|^{\ell(r-1)}\,{\bf 1}\{z\in Z_0|L\}\,\cH^{\ell}(\dint z)\,\nu_\ell(\dint L)\\
&= {\frac{bc}{\ell!}}\left({\frac{\g}{\o_n}}\right)^\ell\EE \int\limits_E\|z\|^{\ell(r-1)}\,{\bf 1}\{z\in Z_0|E\}\,\cH^{\ell}(\dint z)\,.
\end{align*}
The claim follows by writing
\begin{equation*}
\begin{split}
\EE f_{n-\ell}(Z_0) &=  {\frac{bc}{ \ell!}}\left({\frac{\g} {\o_n}}\right)^\ell\,\EE\int\limits_{\SS_E}\int\limits_0^{\r_{Z_0|E}(u)}s^{\ell r-1}\,\dint s\,\cH^{\ell-1}(\dint u)\\
&= {\frac{1} {r\ell}}{\frac{bc}{\ell!}}\left({\frac{\g}{\o_n}}\right)^\ell\,\EE\int\limits_{\SS_E}\r_{Z_0|E}(u)^{\ell r}\,\cH^{\ell-1}(\dint u)\,\\
&= {\frac{1} {r}}{\frac{bc}{\ell!}}\left({\frac{\g}{\o_n}}\right)^\ell\,\EE\tW^{E}_{\ell(1-r)}(Z_0|E)
\end{split}
\end{equation*}
with $c_r(n,\ell)={\frac{1} {r}}{\frac{bc}{\ell!}}\omega_n^{-\ell}$.
\end{proof}

\begin{proof}[Proof of Corollary \ref{VerticeNumberFormula}.]
Using the definition of $c_r(n)$ in the statement of the corollary, the identity for $\ell=n$ in Theorem \ref{thm:ExpectFvector} reads as follows:
$$\EE f_0(Z_0)={\frac{1} {r}}{\frac{1} {n!}}\,\omega_n^{-n}\,c_r(n)\,\g^n\,\EE\tW_{n(1-r)}(Z_0)\,.$$
Using the definition of the dual intrinsic volume we find that
$$\EE\tW_{n(1-r)}(Z_0)={\frac{1} {n}}\EE\int\limits_{\SS^{n-1}}\varrho_{Z_0}(u)^{nr}\,\cH^{n-1}(\dint u)\,.$$
We next use identity \eqref{eq:Momentenformel} and obtain together with the isotropy of $Z_0$ that
$${\frac{1 }{n}}\EE\int\limits_{\SS^{n-1}}\varrho_{Z_0}(u)^{nr}\,\cH^{n-1}(\dint u)={\frac{\o_n}{ n}}\,nr\int\limits_0^\infty s^{nr-1}\,\PP(\varrho_{Z_0}(e_n)>s)\,\dint s\,.$$
Since $Z_0$ is the zero cell of the Poisson hyperplane process $X$, which has intensity measure $\Theta$, we have  $\PP(\varrho_{Z_0}(e_n)>s)=e^{-\Theta\left(A_{[0,se_n]}\right)}$. Lemma \ref{lem:ThetaStrecke0z} then implies that
$$\EE f_0(Z_0)={\frac{1} {n!}}\,\omega_n^{-n}\,c_r(n)\,\g^n\,\o_n\int\limits_0^\infty s^{nr-1}\,e^{-{\frac{2\g}{ r}}{\frac{\o_{r+n}}{\o_n\o_{r+1}}}\,s^r}\,\dint s\,.$$
The proof is completed by
a straightforward integration.
\end{proof}

\begin{proof}[Proof of Corollary \ref{cor:EstimateFVector}]
We start with the lower bound. By Theorem \ref{thm:ExpectFvector} we have
$$
\EE f_{n-\ell}(Z_0)=c_r(n,\ell)\,\g^\ell\,\EE\tW_{\ell(1-r)}^E(Z_0|E)=c_r(n,\ell)\,\g^\ell\,{\frac{1}{\ell}}\EE\int\limits_{\SS_E}\varrho_{Z_0|E}(u)^{\ell r}\,\cH^{\ell-1}(\dint u)
$$
with $E={\rm span}\{e_1,\ldots,e_\ell\}$. Observe that, since $u\in\SS_E$,
$$\varrho_{Z_0|E}(u)\geq\varrho_{Z_0\cap E}(u)=\varrho_{Z_0}(u)\,.$$
For $u\in\SS_E$ and $s\geq 0$, we combine \eqref{eq:Momentenformel} with Lemma \ref{lem:ThetaStrecke0z} to see that
\begin{equation*}
\begin{split}
\EE f_{n-\ell}(Z_0) &\geq c_r(n,\ell)\,\g^\ell\,{\frac{\o_\ell}{\ell}}\,r\ell\,\int\limits_0^\infty s^{r\ell-1}\,\PP(\varrho_{Z_0}(u)>s)\,\dint s\,.
\end{split}
\end{equation*}
The lower bound follows now as in the proof of Corollary \ref{VerticeNumberFormula}.

If $P\subset\R^n$ is a simple polytope, then each vertex is contained in precisely $\binom{n}{j}$ faces of dimension $j$ and each $j$-face has at least $j+1$ vertices. Hence, a simple counting argument yields
\[f_j(P)\leq \frac{1}{j+1}\binom{n}{j}f_0(P),\quad 1\leq j\leq n-1.\]
This implies the upper bound, since the zero cell is almost surely a simple polytope.
\end{proof}

\subsection{Proof of Proposition \ref{Prop:ConstantEstimate}}
Obviously, the relation holds for $\ell=1$.
For $\ell \in \{2,\ldots,n\}$, using spherical coordinates we obtain that
\begin{equation}\label{eq:ZwischschrittAsymptotik1}
\begin{split}
\int\limits_{(\SS^{\ell-1})^\ell} \nabla_\ell(u_1, &\ldots,u_\ell)^{n-\ell+1}\,\prod_{j=1}^\ell|\langle u_j,e_\ell\rangle|^{r-1}\,\cH^{\ell(\ell-1)}\big(\dint(u_1,\ldots,u_\ell)\big)\\
&=\int\limits_{(0,\pi)^\ell}\int\limits_{(\SS^{\ell-1}\cap e_\ell^\bot)^\ell} \nabla_\ell(\cos\theta_1 \,e_\ell+\sin\theta_1\, v_1,\ldots,\cos\theta_\ell\, e_\ell+\sin\theta_\ell\, v_\ell)^{n-\ell+1}\\
&\qquad\times \prod\limits_{j=1}^\ell \big(|\cos\theta_j|^{r-1} (\sin\theta_j)^{\ell-2}\big)\, \cH^{\ell(\ell-2)}\big(\dint(v_1,\ldots,v_\ell)\big)\,\dint(\theta_1\ldots\theta_\ell)\,.
\end{split}
\end{equation}
Next we use multilinearity and Laplace's expansion for the determinant to see that the last expression equals
\begin{equation}\label{eq:ZwischenschrittAsymptotik2}
\begin{split}
&2^\ell \int\limits_{(0,\frac{\pi}{2})^\ell}\int\limits_{(\SS^{\ell-1}\cap e_\ell^\bot)^\ell}
\Big|\sum\limits_{j=1}^\ell (-1)^j \cos\theta_j \sin\theta_1\cdots\sin\theta_{j-1}\sin\theta_{j+1}
\cdots \sin\theta_{\ell}\\
&\qquad\times\det(v_1,\ldots,v_{j-1},v_{j+1},\ldots,v_{\ell})\Big|^{n-\ell+1} \prod\limits_{j=1}^\ell\big( (\cos\theta_j)^{r-1} (\sin\theta_j)^{\ell-2}\big)\\
&\qquad\qquad\cH^{\ell(\ell-2)}\big(\dint(v_1,\ldots,v_\ell)\big)\,\dint(\theta_1\ldots\theta_\ell)\\
& =2^\ell \int\limits_{(0,\frac{\pi}{2})^\ell}\int\limits_{(\SS^{\ell-1}\cap e_\ell^\bot)^\ell}
\left|\sum\limits_{j=1}^\ell (-1)^j\, \frac{\cos\theta_j}{\sin\theta_j} \det(v_1,\ldots,v_{j-1},v_{j+1},\ldots,v_{\ell})\right|^{n-\ell+1}\\
&\qquad\times \prod\limits_{j=1}^\ell\big( (\cos\theta_j)^{r-1} (\sin\theta_j)^{n-1}\big)\, \cH^{\ell(\ell-2)}\big(\dint(v_1,\ldots,v_\ell)\big)\,\dint(\theta_1\ldots\theta_\ell)\,.
\end{split}
\end{equation}
Using H\"older's inequality, it follows for a  measurable function
$f:\SS^{\ell-1}\cap e_\ell^\bot\rightarrow \mathbb{R}$ that
\[
\left(\;\int\limits_{\SS^{\ell-1}\cap e_\ell^\bot} |f(v)|^{n-\ell+1}\,\cH^{\ell-2}(\dint v)\right)^{\frac{1}{n-\ell+1}} \geq
\omega_{\ell-1}^{-{\frac{n-\ell}{ n-\ell+1}}}
\int\limits_{\SS^{\ell-1}\cap e_\ell^\bot} |f(v)|\;\cH^{\ell-2}(\dint v)\,,
\]
which implies that
\begin{equation*}
\begin{split}
&2^\ell\int\limits_{(0,\frac{\pi}{2})^\ell}\int\limits_{(\SS^{\ell-1}\cap e_\ell^\bot)^\ell}
\left|\sum\limits_{j=1}^\ell (-1)^j\, \frac{\cos\theta_j}{\sin\theta_j} \det(v_1,\ldots,v_{j-1},v_{j+1},\ldots,v_{\ell})\right|^{n-\ell+1}\\
&\qquad\times \prod\limits_{j=1}^\ell\big( (\cos\theta_j)^{r-1} (\sin\theta_j)^{n-1}\big)\, \cH^{\ell(\ell-2)}\big(\dint(v_1,\ldots,v_\ell)\big)\,\dint(\theta_1\ldots\theta_\ell)\\
& =2^\ell\int\limits_{(0,\frac{\pi}{2})^\ell}\int\limits_{(\SS^{\ell-1}\cap e_\ell^\bot)^{\ell-1}} \left(\;\int\limits_{\SS^{\ell-1}\cap e_\ell^\bot}
\left|\sum\limits_{j=1}^\ell (-1)^j\, \frac{\cos\theta_j}{\sin\theta_j} \det(v_1,\ldots,v_{j-1},v_{j+1},\ldots,v_{\ell})\right|^{n-\ell+1} \cH^{\ell-2}(\dint v_1)\right)\\
&\qquad\times \prod\limits_{j=1}^\ell \big((\cos\theta_j)^{r-1} (\sin\theta_j)^{n-1}\big)\, \cH^{(\ell-1)(\ell-2)}\big(\dint(v_2,\ldots,v_\ell)\big)\,\dint(\theta_1\ldots\theta_\ell)\\
&\geq 2^\ell\int\limits_{(0,\frac{\pi}{2})^\ell}\int\limits_{(\SS^{\ell-1}\cap e_\ell^\bot)^{\ell-1}}
\omega_{\ell-1}^{-(n-\ell)} \\
&\qquad\times \left(\;\int\limits_{\SS^{\ell-1}\cap e_\ell^\bot}
\left|\sum\limits_{j=1}^\ell (-1)^j\, \frac{\cos\theta_j}{\sin\theta_j} \det(v_1,\ldots,v_{j-1},v_{j+1},\ldots,v_{\ell})\right|\cH^{\ell-2}(\dint v_1)\right)^{n-\ell+1}\\
&\qquad\times \prod\limits_{j=1}^\ell \big((\cos\theta_j)^{r-1} (\sin\theta_j)^{n-1}\big) \, \cH^{(\ell-1)(\ell-2)}\big(\dint(v_2,\ldots,v_\ell)\big)\,\dint(\theta_1\ldots\theta_\ell)\,.
\end{split}
\end{equation*}
This is made smaller by taking the inner integral into the absolute values, which yields
\begin{equation*}
\begin{split}
&\int\limits_{(\SS^{\ell-1})^\ell} \nabla_\ell(u_1,\ldots,u_\ell)^{n-\ell+1}\,\prod_{j=1}^\ell\left|\langle u_j,e_\ell\rangle\right|^{r-1}\,\cH^{\ell(\ell-1)}\big(\dint(u_1,\ldots,u_\ell)\big)\\
&\geq 2^\ell\int\limits_{(0,\frac{\pi}{2})^\ell}\int\limits_{(\SS^{\ell-1}\cap e_\ell^\bot)^{\ell-1}}
\omega_{\ell-1}^{-(n-\ell)} \left|
\sum\limits_{j=1}^\ell (-1)^j\, \frac{\cos\theta_j}{\sin\theta_j} \int\limits_{\SS^{\ell-1}\cap e_\ell^\bot} \det(v_1,\ldots,v_{j-1},v_{j+1},\ldots,v_{\ell})\,\cH^{\ell-2}(\dint v_1)\right|^{n-\ell+1}\\
&\qquad\times \prod\limits_{j=1}^\ell \big((\cos\theta_j)^{r-1} (\sin\theta_j)^{n-1}\big) \cH^{(\ell-1)(\ell-2)}\big(\dint(v_2,\ldots,v_\ell)\big)\,\dint(\theta_1\ldots\theta_\ell)\\
& =2^\ell\,\omega_{\ell-1} \left(\;\int\limits_{0}^{\frac{\pi}{2}}(\cos\theta_1)^{n-\ell+r}\,(\sin\theta_1)^{\ell-2}\,\dint\theta_1\right)\, \left(\int\limits_0^{\frac{\pi}{2}}(\cos\theta)^{r-1}\,(\sin\theta)^{n-1}\,\dint\theta\right)^{\ell-1}\\
&\qquad\times\int\limits_{(\SS^{\ell-1}\cap e_\ell^\bot)^{\ell-1}}\nabla_{\ell-1}(v_1,\ldots,v_{l-1})^{n-\ell+1}\,
\cH^{(\ell-1)(\ell-2)}\big(\dint(v_1,\ldots,v_{\ell-1})\big)\,,
\end{split}
\end{equation*}
since
$$\sum\limits_{j=1}^\ell (-1)^j\, \frac{\cos\theta_j}{\sin\theta_j} \int\limits_{\SS^{\ell-1}\cap e_\ell^\bot} \det(v_1,\ldots,v_{j-1},v_{j+1},\ldots,v_{\ell})\,\cH^{\ell-2}(\dint v_1)=-\frac{\cos\theta_1}{\sin\theta_1}\, \omega_{\ell-1} \det(v_2,\ldots,v_{\ell})\,.$$
The integrals in the first two brackets can be evaluated directly using \eqref{eqGamma}. They equal
$${\frac{\o_{n+r}}{\o_{n-\ell+r+1}\o_{\ell-1}}}\qquad{\rm and}\qquad\left({\frac{\o_{n+r}}{\o_r\o_{n}}}\right)^{\ell-1},$$
respectively. The remaining factor can be treated by means of the Blaschke-Petkanschin formula \cite[Theorem 7.2.3]{SW} (applied backwards with $d=n$, $p=q=\ell-1$, $r_1=\ldots=r_{\ell-1}=1$ and $f\equiv 1$ there). This gives
$$\int\limits_{(\SS^{\ell-1}\cap e_\ell^\bot)^{\ell-1}}\nabla_{\ell-1}(v_1,\ldots,v_{\ell-1})^{n-\ell+1}\,\cH^{(\ell-1)(\ell-2)}\big(\dint(v_1,\ldots,v_{\ell-1})\big)=\o_{n}^{\ell-1}\,{\frac{\o_1\ldots\o_{\ell-1}}{\o_{n-\ell+2}\ldots\o_n}}\,.$$
Now the lower estimate follows immediately.

To obtain an upper bound, we use the inequality
$$\sum\limits_{j=1}^\ell |x_j| \leq  \ell^{\frac{n-\ell}{n-\ell+1}}\left(\sum\limits_{j=1}^\ell |x_j|^{n-\ell+1}\right)^{\frac{1}{n-\ell+1}}\,,$$
valid for arbitrary real numbers $x_1,\ldots,x_\ell$ as a consequence of H\"older's inequality, to see that
\begin{align}
&\int\limits_{(0,\frac{\pi}{2})^\ell}\int\limits_{(\SS^{\ell-1}\cap e_\ell^\bot)^\ell}
\left|\sum\limits_{j=1}^\ell (-1)^j\, \frac{\cos\theta_j}{\sin\theta_j} \det(v_1,\ldots,v_{j-1},v_{j+1},\ldots,v_{\ell})\right|^{n-\ell+1}\nonumber\\
&\qquad\times \prod\limits_{j=1}^\ell\big( (\cos\theta_j)^{r-1} (\sin\theta_j)^{n-1}\big)\, \cH^{\ell(\ell-2)}\big(\dint(v_1,\ldots,v_\ell)\big)\, \dint(\theta_1\ldots\theta_\ell)\nonumber\\
&\leq \ell^{n-\ell}\,\int\limits_{(0,\frac{\pi}{2})^\ell} \int\limits_{(\SS^{\ell-1}\cap e_\ell^\bot)^\ell}
\sum\limits_{j=1}^\ell  \left|\frac{\cos\theta_j}{\sin\theta_j} \det(v_1,\ldots,v_{j-1},v_{j+1},\ldots,v_{\ell})\right|^{n-\ell+1}
\nonumber\\
&\qquad\times \prod\limits_{j=1}^\ell\big( (\cos\theta_j)^{r-1} (\sin\theta_j)^{n-1}\big)\, \cH^{\ell(\ell-2)}\big(\dint(v_1,\ldots,v_\ell)\big)\dint(\theta_1\ldots\theta_\ell)\,.\label{Eq:Upperbound}
\end{align}
To complete the proof of the upper bound, we combine \eqref{eq:ZwischschrittAsymptotik1} with \eqref{eq:ZwischenschrittAsymptotik2} and then, the inequality \eqref{Eq:Upperbound}
 eventually leads to the estimate $c_r(n,\ell)\leq\ell^{n-\ell+1}\,A(n,\ell,r)$, which completes the proof.\hfill $\Box$

\bigskip

{\em Proof of Corollary \ref{cor:EndabschaetzungFVektor}.}
The bounds for $\mathbb{E}f_{0}(Z_0)$ follow from Corollary \ref{VerticeNumberFormula} and the bounds for $c_r(n,n)$ in Proposition
\ref{Prop:ConstantEstimate} if $c_r(n)=r n! \omega_n^n c_r(n,n)$ is used.

For $\ell\in\{1,\ldots,n-1\}$ the bounds for
$\mathbb{E}f_{n-\ell}(Z_0)$ are obtained by substituting the bounds for $c_r(n,\ell)$ from Proposition \ref{Prop:ConstantEstimate}
into the bounds from Corollary \ref{cor:EstimateFVector} and by using the upper bound for  $\mathbb{E}f_{0}(Z_0)$.\hfill $\Box$

\subsection{Proof of Theorem \ref{thm:Skelett}}

Arguing as at the beginning of the proof of Theorem \ref{thm:ExpectFvector}, we see that
\[
\EE\mathcal{H}^{n-\ell} (\text{skel}_{n-\ell}( Z_0)) = \frac{1}{\ell!}\,\mathbb{E}\sum_{(H_1,\ldots,H_\ell)\in X^\ell_{\neq}}\mathcal{H}^{n-\ell}(Z_0\cap H_1\cap \ldots \cap H_\ell)\,.
\]
We can proceed as in the first part of the proof of Theorem \ref{thm:ExpectFvector} and obtain with $b$ as at \eqref{eq:DefConstantB} that
\begin{align*}
&\EE\mathcal{H}^{n-\ell}(\text{skel}_{n-\ell}( Z_0))\\
& =  b{\frac{1}{\ell!}}\left({\frac{\g}{\o_n}}\right)^\ell\EE\int\limits_{G(n,\ell)}\int\limits_{(\SS_L)^\ell} \int\limits_L\nabla_\ell(u_1,\ldots,u_\ell)^{n-\ell+1}\,\cH^{n-\ell}((u_1^\perp+z)\cap\ldots\cap(u_\ell^\perp+z)\cap Z_0)\\
&\hspace{3cm}\times\prod_{j=1}^\ell|\langle u_j,z\rangle|^{r-1}\,\cH^\ell(\dint z)\,\cH^{\ell(\ell-1)}\big(\dint(u_1,\ldots,u_\ell)\big)\,\nu_\ell(\dint L)\\
& =  b{\frac{1}{\ell!}}\left({\frac{\g}{\o_n}}\right)^\ell\EE\int\limits_{G(n,\ell)}\int\limits_{(\SS_L)^\ell}\int\limits_L\nabla_\ell(u_1,\ldots,u_\ell)^{n-\ell+1}\,\cH^{n-\ell}((z+L^\perp)\cap Z_0)\\
&\hspace{3cm}\times\prod_{j=1}^\ell|\langle u_j,z\rangle|^{r-1}\,\cH^\ell(\dint z)\,\cH^{\ell(\ell-1)}\big(\dint(u_1,\ldots,u_\ell)\big)\,\nu_\ell(\dint L).
\end{align*}
Next we argue as in the second part of the proof of Theorem \ref{thm:ExpectFvector}. Then we arrive at
\begin{align*}
&\EE\mathcal{H}^{n-\ell}(\text{skel}_{n-\ell}( Z_0))
 = r\,c_r(n,\ell) \, \gamma^\ell\,\EE \int\limits_E\|z\|^{\ell(r-1)}\, \cH^{n-\ell}\big((z+E^\perp)\cap Z_0\big) \,\cH^{\ell}(\dint z)\,.
\end{align*}
Using Fubini's Theorem, spherical coordinates, Lemma \ref{lem:ThetaStrecke0z} and \eqref{eqGamma}, we get
\begin{align*}
\EE\mathcal{H}^{n-\ell}(\text{skel}_{n-\ell}( Z_0))
& = r\, c_r(n,\ell) \,\gamma^\ell \, \int\limits_{E}\|z\|^{\ell(r-1)}\int\limits_{E^\perp} \PP(x+z\in Z_0) \, \dint x \, \dint z\\
& =r\,c_r(n,\ell) \,\gamma^\ell \, \o_\ell\o_{n-\ell} \int\limits_0^\infty \int\limits_0^\infty s^{\ell r-1}\,
 t^{n-\ell-1} \,\PP\big(\sqrt{s^2+t^2}\; e_1\in Z_0\big)\, \dint t\,\dint s\\
 & = r\,c_r(n,\ell) \,\gamma^\ell \,\o_\ell\o_{n-\ell} \int\limits_0^\infty e^{-{\frac{2\gamma}{  r}}{\frac{\o_{r+n}} {\o_n \o_{r+1}}}x^r}\,
x^{\ell r+n-\ell-1}\, \dint x \;\int\limits_0^{\pi/2}(\cos\theta )^{\ell r-1} (\sin\theta)^{n-\ell-1} \,\dint\theta\\
& = 2\,c_r(n,\ell) \,\gamma^{-{\frac{n-\ell}{ r}}} \, {\frac{\o_\ell \o_{\ell r+n-\ell}}{ \o_{\ell r} \o_{2(\ell+{\frac{n-\ell}{  r}})}}}
\left(r\,{\frac{\pi}{ 2}}{\frac{\o_n\o_{r+1} } {\o_{r+n}}}\right)^{\ell+{\frac{n-\ell}{ r}}}.
\end{align*}
The statement of the theorem is thus proved.\hfill $\Box$

\subsection{Proofs of the formulae for $r=1$}

\begin{proof}[Proof of Theorem \ref{thm:VereinheitnlichungSchneiderWieacker}]
For $\ell=0$ the relation is obvious. For $\ell\in\{1,\ldots,n\}$ we first use the multivariate Mecke formula \eqref{eq:Mecke} and then $\ell$-times Crofton's formula \cite[Theorem 5.1.1]{SW} to see that
\begin{align*}
\EE F_{n-\ell;j}(Z_0) &= {\frac{1}{ \ell!}}\,\EE\sum_{(H_1,\ldots,H_\ell)\in X_{\neq}^\ell}V_j(Z_0\cap H_1\cap\ldots\cap H_\ell)\\
&= {\frac{\g^\ell}{ \ell!}}\,\EE\int\limits_{A(n,n-1)^\ell}V_j(Z_0\cap H_1\cap\ldots\cap H_\ell)\,\mu_{n-1}^\ell\big(\dint(H_1\ldots\dint H_\ell)\big)\\
&= {\frac{\g^\ell}{ \ell!}}\,\left({\frac{\kappa_{n-1} } {\o_n}}\right)^\ell\,{\frac{(\ell+j)! } {j!}}\,{\frac{\kappa_{\ell+j}}{\kappa_j}}\,\EE V_{\ell+j}(Z_0)\,,
\end{align*}
where the measure $\mu_{n-1}$ has been defined in \eqref{eq:DefinitionMuEll}. This proves the claim.
\end{proof}

\begin{proof}[Proof of Corollary \ref{cor:VergleichfF}]
This follows by comparing the expression \eqref{eq:AllgemeineSchneiderformel} for $\EE f_{n-\ell-j}$ with that of $\EE F_{n-\ell;j}$ given by Theorem \ref{thm:VereinheitnlichungSchneiderWieacker}.
\end{proof}

\begin{proof}[Proof of Corollary \ref{cor:NiederdimensionaleSeitenF}]
By \eqref{eq:lSeiteSchnittdarstellung}, for fixed $L\in G(n,m)$ the $m$-dimensional random polyhedron $Z_0\cap L$ is almost surely the zero cell of $X\cap L$, the random tessellation induced by the intersection of $X$ with $L$. This sectional tessellation has intensity $\gamma_L$ given by
$$\gamma_L={\frac{\o_m\omega_{n+1}}{\omega_n\omega_{m+1}}} \gamma, 
$$
independently of the subspace $L$, as a consequence of the rotation invariance of $X$, cf.\ \cite[Equation (3.29T)]{MilesFlats} or Proposition \ref{prop:IntensityMeasureSectionWithPlane} with $r=1$ there. Thus, applying Theorem \ref{thm:VereinheitnlichungSchneiderWieacker} to the zero cell $Z_0\cap L$ of $X\cap L$ and combining with \eqref{eq:lSeiteSchnittdarstellung}, we get
$$
\EE F_{m-j;i}(Z_m) =\gamma_L^j\,\binom{i+j}{ j}\,\left({\frac{\kappa_{m-1}}{\o_m}}\right)^j\,{\frac{\kappa_{i+j}}{\kappa_i}}\,\int\limits_{G(n,m)}\EE V_{i+j}(Z_0\cap L)\,\nu_m(\dint L)\,,
$$
which in view of \eqref{eq:lSeiteSchnittdarstellung} is -- after simplification of the constants -- the formula in Corollary \ref{cor:NiederdimensionaleSeitenF}.
\end{proof}

\begin{proof}[Proof of Theorem \ref{thm:HoehereMomente}]
Let us introduce the abbreviation $\psi(Z_m)={\frac{2\kappa_{n-1}} {\o_n}}V_1(Z_m)$ and apply Theorem 6.1 in \cite{BL} to deduce that the conditional distribution of $\psi(Z_m)$, given $f_{m-1}(Z_m)=i$ for some integer $i\geq m+1$, is a Gamma distribution with parameters $i$ and $\g$ and mean $\frac{i}{ \gamma}$. Thus, we get
$$\EE\big[\psi(Z_m)^k\,|\,f_{m-1}(Z_m)=i\big]=\g^{-k} \prod\limits_{q=0}^{k-1}(i+q),\qquad k\in \NN.$$
Consequently, writing $p_i$ for the probability that $f_{m-1}(Z_m)=i$, we conclude that
\begin{equation}\label{eq:ZwischenschrittHoehereMomente2}
\EE\big[\psi(Z_m)^k\big]=\sum_{i=m+1}^\infty p_i\,\g^{-k}\,\prod_{q=0}^{k-1}(i+q),\qquad k\in \NN.
\end{equation}
Now we observe that
$$\prod_{q=0}^{k-1}(i+q)=\sum_{q=1}^k \begin{bmatrix}k\\ q\end{bmatrix} i^q\,,$$
which in view of \eqref{eq:ZwischenschrittHoehereMomente2} implies that
$$\EE\psi(Z_m)^k=\g^{-k}\,\sum_{q=1}^k\begin{bmatrix}k\\ q\end{bmatrix}\,\sum_{i=m+1}^\infty p_i\,i^q=\g^{-k}\,\sum_{q=1}^k\begin{bmatrix}k\\ q\end{bmatrix}\,\EE f_{m-1}^q(Z_m)\,.$$
Substituting finally the expression for $\psi(Z_m)$, we complete the proof.
\end{proof}

\subsection{Proof of Proposition \ref{prop:IntensityMeasureSectionWithPlane}}

By definition of $X\cap L$ we have
\begin{align*}
& \Theta_{ L}(\,\cdot\,) = \frac{\gamma}{\o_n} \int\limits_{\SS^{n-1}}\int\limits_{\RR} \mathbf{1}\{H(u,t)\cap L \in \,\cdot\,\}\, |t|^{r-1} \, \dint t\, \cH^{n-1}(\dint u)\,.
\end{align*}
The map
\[
F:\begin{cases}
\SS_L\times(0,\frac{\pi}{2})\times\SS_{L^\perp} &\to\quad \SS^{n-1}\\
(u_1,\theta,u_2)&\mapsto\quad \cos(\theta) u_1 + \sin(\theta)u_2
\end{cases}
\]
is injective and its image covers $\SS^{n-1}$ up to a set of measure zero. Its Jacobian is
\[JF(u_1,\theta,u_2)= (\cos \theta)^{m-1} (\sin \theta)^{n-m-1}.\]
Thus,
\begin{align*}
\Theta_{ L}(\,\cdot\,)
& = \frac{\gamma}{\o_n} \int\limits_{\SS_{L^\bot}}\int\limits_{\SS_L} \int\limits_{0}^{\frac{\pi}{2}} \int\limits_{\RR} \mathbf{1}\big(H(\cos(\theta)u_1 + \sin(\theta)u_2,t)\cap L \in \,\cdot\,\big)\, |t|^{r-1} \, \dint t\\
& \qquad\qquad \times (\cos\theta)^{m-1}
(\sin\theta)^{n-m-1}\,\dint\theta\, \cH^{m-1}(\dint u_1)\, \cH^{n-m-1}(\dint u_2)\,.
\end{align*}
Moreover, a short computation shows that
\[
H(\cos(\theta)u_1 + \sin(\theta)u_2,t)\cap L = \left(\frac{t}{\cos(\theta)}\right)u_1 + (u_1^\bot \cap L)\,,
\]
which implies together with \eqref{eqGamma} that
\begin{align*}
\Theta_{ L}(\,\cdot\,)
& = \frac{\gamma}{\o_n} \int\limits_{\SS_{L^\bot}}\int\limits_{\SS_{L}} \int\limits_0^{\frac{\pi}{2}} \int\limits_{\RR} \mathbf{1}\left(\left(\frac{t}{\cos(\theta)}\right)u_1 + (u_1^\bot \cap L) \in \,\cdot\,\right)\, |t|^{r-1}\, \dint t\\
& \qquad\qquad \times (\cos\theta)^{m-1}(\sin\theta)^{n-m-1}\,\dint\theta\, \cH^{m-1}(\dint u_1)\, \cH^{n-m-1}(\dint u_2)\\
& =  \frac{\gamma}{\o_n} \int\limits_{\SS_{L^\bot}}\int\limits_{\SS_{L}} \int\limits_0^{\frac{\pi}{2}} \int\limits_{\RR} \mathbf{1}\big(t u_1 + (u_1^\bot \cap L) \in \,\cdot\,\big)\, |t|^{r-1} \,\dint t\\
& \qquad\qquad \times (\cos\theta)^{m+r-1}(\sin\theta)^{n-m-1} \,\dint\theta\, \cH^{m-1}(\dint u_1)\, \cH^{n-m-1}(\dint u_2)\\
& = \frac{\gamma}{\o_n} \frac{\o_{n+r}}{\o_{m+r}}
\int\limits_{\SS_{L}} \int\limits_{\RR} \mathbf{1}\big(t u + (u^\bot \cap L) \in \,\cdot\,\big)\, |t|^{r-1} \, \dint t\,  \cH^{m-1}(\dint u)\,.
\end{align*}
This completes the proof. \hfill $\Box$

\subsection{Proofs related to Section \ref{subsec:asymptotikundhyperplane}}

\begin{proof}[Proof of Corollary \ref{cor:AsymptotikVolumeSection}]
Stirling's formula states that
\begin{equation}\label{eq:Stirling}\Gamma(x)=\sqrt{\frac{2\pi} {x}}\left({\frac{x}{ e}}\right)^x\,e^{\frac{\lambda(x)}{12x}}\end{equation}
with $\lambda(x)\in(0,1)$ for all $x>0$ (see \cite[Equation (12.33)]{WW} or \cite[p.\ 24]{Artin1}). For two expressions $A(n), B(n)$, depending on $n$, we write $A(n)\sim B(n)$ as $n\to\infty$ if $\lim\limits_{n\to\infty}\frac{A(n)}{B(n)}=1.$
The limiting relation for $\EE V_{n-\ell}(Z_0\cap L)$ follows from Proposition \ref{prop:BoundsMomentsVolmeSections} by choosing the intensity as $\widehat{\gamma}(r,n)$. For $r=b\, n^\alpha$ with $b>0$ and $\alpha\in\RR$, we apply Equation \eqref{eq:Stirling} and use for $\alpha\geq 1$ the continuity of the gamma function. For fixed $\ell\in \N$ and $L\in G(n,n-\ell)$, as $n\to\infty$ we obtain
\begin{align*}
\EE V_{n-\ell}(Z_0\cap L)&=\frac{\G(\frac{n-\ell}{r}+1)}{\G(\frac{n-\ell}{2}+1)}
\left(\frac{\G(\frac{n}{2}+1)}{\G(\frac{n}{r}+1)}\right)^{\frac{n-\ell}{n}}\sim
\bigg(\Big(1-\frac{\ell}{n}\Big)^{(n-\ell)}\bigg)^{\frac{1}{r}}
\bigg(1-\frac{\ell}{n}\bigg)^{-\frac{n-\ell}{2}}\\
&
\to\begin{cases} 0&: \alpha<0\\
e^{-\frac{\ell}{b}+\frac{\ell}{2}}&:\alpha=0\\
e^{\ell/2}&:\alpha>0\,.\\
\end{cases}
\end{align*}
 To see the variance bound, we analyse the behaviour of the quantities $D(n,n-\ell,r)$ and $E(n-\ell,r)$ occurring in Proposition \ref{prop:BoundsVarianceSections} with $m=n-\ell$.

 To the constant $D(n,n-\ell,r)$ defined in  \eqref{AuxiliaryConstantD(n,m,r)} we apply Stirling's formula and for $\alpha\geq 1$ we use the continuity of the gamma function on $(0,\infty)$. Thus, for fixed $\ell\in\N$ and $L\in G(n,n-\ell)$, as $n\to\infty$ we deduce  that
\begin{align*}
&D(n,n-\ell,r)= \frac{n-\ell}{r}\left(\frac{\Gamma(\frac{n}{2}+1)}
{\Gamma(\frac{n}{r}+1)}\right)^{\frac{2(n-\ell)}{n}} \frac{\G(\frac{2(n-\ell)}{r}+1)}{(\G(\frac{n-\ell}{2}+1))^2}\; 2^{-\frac{2(n-\ell)}{r}}\\[0.3cm]
%
%
&\sim\begin{cases} \frac{1}{\sqrt{\pi}}\frac{\sqrt{n}}{\sqrt{r}} (1-\frac{\ell}{n})^{-(n-\ell)}\left(\left( 1-\frac{\ell}{n}\right)^{2(n-\ell)}\right)^{\frac{1}{r}}&: \alpha<1\\[0.3cm]
\left( 1-\frac{\ell}{n}\right)^{-(n-\ell)}\,\left( \frac{n-\ell}{r}\right)\,2^{-2\frac{(n-\ell)}{r}}\, \frac{\G(\frac{2(n-\ell)}{r}+1)} {(\G(\frac{n}{r}+1))^{\frac{2(n-\ell)}{n}}}&: \alpha\geq 1\\
\end{cases}\\[0.3cm]
%
%
&\quad\,\begin{cases}
\in \left[\left(e^{-\frac{2}{b}}-\epsilon\right)^{\ell\,n^{-\alpha}},
\left(e^{-\frac{2}{b}}+\epsilon\right)^{\ell\,n^{-\alpha}}\right]&: \alpha<0\text{ and } \ell>0\vspace{0.2cm}\\
\sim c_1\,n^{\frac{1-\alpha}{2}}&:\alpha<0 \text{ and }\ell=0\\ & \,\text{ or }0\leq \alpha<1\\
\to e^\ell \,\frac{\G(\frac{2}{b})}{(\G(\frac{1}{b}))^2}\,2^{1-\frac{2}{b}}&: \alpha=1 \vspace{0.2cm}\\
\sim c_2\, n^{1-\alpha}&:\alpha>1\,,
\end{cases}
\end{align*}
where $c_1,c_2>0$ and $\epsilon\in \left(0,e^{-\frac{2}{b}}\right)$ are constants not depending on $n$ and $n\geq N_{\e}$, for some $N_{\e}\in \N$.
From inequality \eqref{BoundForE} we obtain
\begin{align*}
c \left( 1+\frac{r}{n-\ell}\right)^{-\frac{1}{2}}F(n-\ell,r)\leq E(n-\ell,r)\leq C\left( 1+\frac{r}{n-\ell}\right)F(n-\ell,r),
\end{align*}
with constants $c, C>0$, independent of $r$ and $n$, and
\begin{align*}
&F(n-\ell,r)=\frac{1}{\sqrt{r+1}}\,2^{\frac{n-\ell}{2}}\, \left(1+\frac{r}{2(n-\ell)}\right)^{-\frac{n-\ell}{2}}
\left(1+\frac{n-\ell}{n-\ell+r}\right)^{-\frac{n-\ell+r}{2}}\\[0.5cm]
&\begin{cases}
=\frac{1}{\sqrt{r+1}}\left(\frac{1}{2}\left( 1+\frac{1}{2\frac{n-\ell}{r}}\right)^{-\frac{n-\ell}{r}} \left( 1-\frac{1/2}{1+\frac{n-\ell}{r}}\right)^{-1-\frac{n-\ell}{r}}\right)^{\frac{r}{2}}&:\alpha\leq 1\\[0.4cm]
\sim c_3\; \frac{1}{\sqrt{r}}
\left(\frac{n}{r}\;4 \;\left( 1+\frac{2 (n-\ell)}{r}\right)^{-1} \left( 1+\frac{1}{1+\frac{r}{n-\ell}}\right)^{-1-\frac{r}{n-\ell}} \right)^{\frac{n-\ell}{2}} &:\alpha>1
\end{cases}\\[0.3cm]
&\begin{cases} \sim c_4 n^{-\frac{\alpha}{2}}&:\alpha<0\\
\to \left((b+1)\,2^b\right)^{-\frac{1}{2}} &:\alpha =0\\
\in \left[(2^{-\frac{b}{2}}-\epsilon)^{n^\alpha} ,(2^{-\frac{b}{2}}+\epsilon)^{n^\alpha} \right] &: 0<\alpha<1\\
\sim c_5\,\frac{1}{\sqrt{n}}\left( \frac{4(b+1)^{b+1}}{(b+2)^{b+2}}\right)^{\frac{n}{2}}&:\alpha=1\\[0.2cm]
\in \left[\left(n^{1-\alpha}\frac{4}{be}-\epsilon\right)^{\frac{n}{2}}
,\left(n^{1-\alpha}\frac{4}{be}+\epsilon\right)^{\frac{n}{2}}\right] \hspace{2.7cm}&:\alpha>1
\end{cases}
\end{align*}
for constants $c_3,c_4,c_5>0$ and $\epsilon\in (0, 2^{-\frac{b}{2}})$ not depending on $n$ and $n\geq N_{\e}$, for some $N_{\e}\in \N$.
The limiting behaviour of the additional factor in the upper bound of Proposition \ref{prop:BoundsVarianceSections} is
\begin{align*}
4^{\frac{2(n-\ell)}{r}+1}\,
\begin{cases}\sim \left( 4^{\frac{2}{b}}\right)^{n^{1-\alpha}}&:\alpha<1\\
\to 4^{\frac{2}{b}+1}&:\alpha=1\\
\to 4&: \alpha>1.\end{cases}
\end{align*}
For the behaviour of $\var(V_{n-\ell}(Z_0\cap L))$ we thus obtain

\begin{align*}
\var(V_{n-\ell}(Z_0\cap L))&\,
\begin{cases}
\in\left[ c_6\, n^{\frac{1}{2}-\a},c_7\, n^{\frac{1}{2}-\a}\left( 4^{\frac{2}{b}}\right)^{n^{1-\a}}\right]
&:\alpha<0 \text{ and } \ell =0\\[0.5cm]
\in \left[ \left( e^{-\frac{2}{b}}-\e\right)^{\ell\, n^{-\a}}, \left( 4^{\frac{2}{b}}\right)^{n^{1-\a}}\right]&:\alpha<0 \text{ and }\ell>0 \\[0.5cm]
\in \left[c_8 \sqrt{n},c_9\sqrt{n}\left( 4^{\frac{2}{b}}\right)^n\right]&:\alpha=0 \\[0.5cm]
\in\left[ \left( 2^{-\frac{b}{2}}-\e\right)^{n^\a},\left(2^{-\frac{b}{2}}+\e\right)^{n^\a}\left( 4^{\frac{2}{b}}\right)^{n^{1-\a}}\right]&:0<\alpha<\frac{1}{2} \\[0.5cm]
\in\left[\left(2^{-\frac{b}{2}}-\e\right)^{\sqrt{n}},
\left(2^{\frac{4}{b}-\frac{b}{2}}+\e\right)^{\sqrt{n}}\right]&:\alpha=\frac{1}{2} \\[0.5cm]
\in\left[\left(2^{-\frac{b}{2}}-\e\right)^{n^\a},
\left(2^{-\frac{b}{2}}+\e\right)^{n^\a}\right]&:\frac{1}{2}<\alpha<1  \\[0.5cm]
\sim c_{10}\,\frac{1}{\sqrt{n}}\left( \frac{4(b+1)^{b+1}}{(b+2)^{b+2}}\right)^{\frac{n}{2}} &:\alpha=1 \\[0.5cm]
\in \left[ \left( n^{1-\a} \frac{4}{be}-\e\right)^{\frac{n}{2}}, \left( n^{1-\a} \frac{4}{be}+\e\right)^{\frac{n}{2}}\right]&:\alpha>1 \\
\end{cases}
\end{align*}
 for constants $c_6,c_7,c_8,c_9,c_{10}>0$ and $\e\in \left(0,e^{-\frac{2}{b}}\right)$ not depending on $n$ and $n\geq N_{\e}$, for some $N_{\e}\in \N$.
 Therefore, if $\alpha<0$ and $l=0$, or if $\a=0$, then   $\var(V_{n-\ell}(Z_0\cap L))$ goes to infinity as $n$ goes to infinity. Observe that
 $\frac{4}{b}-\frac{b}{2}\geq 0$ is equivalent to $b\leq \sqrt{8}$.
Thus, the limiting behaviour of $\var(V_{n-\ell}(Z_0\cap L))$ remains open if $\alpha<0$ and $\ell>0$, or if $\a\in \left(0,\frac{1}{2}\right)$, or if $\alpha=\frac{1}{2}$ and $b\in (0,\sqrt{8}]$.
Finally, $\var(V_{n-\ell}(Z_0\cap L))$  converges to zero as $n$ goes to infinity if $\a=\frac{1}{2}$ and $b>\sqrt{8}$, or if $\a>\frac{1}{2}$.
Together this yields the assertion.
\end{proof}

\begin{proof}[Proof of Theorem \ref{thm:HyperplaneConjecture}]
By our choice of the intensity we have $\mathbb{E}V_n(Z_0) = 1$ and by Corollary \ref{cor:AsymptotikVolumeSection} it holds that $\lim_{n\rightarrow \infty}\mathbb{E}V_{n-1}(Z_0 \cap L)= \sqrt{e}$.
Hence, for any $\varepsilon\in(0,\sqrt{e})$ there exists an $N_\varepsilon \in \mathbb{N}$ such that for space dimensions $n \geq N_\varepsilon$ we have
\[
\mathbb{E}V_{n-1}(Z_0 \cap L) \in \left(\sqrt{e}-\frac{\varepsilon}{4},\sqrt{e}+\frac{\varepsilon}{4}\right)\,.
\]
For such $n \geq N_\varepsilon$ we can write
\begin{align*}
& \mathbb{P}\left(V_{n-1}(\overline{Z}_0 \cap L) > \sqrt{e}-\varepsilon\right)
\geq \mathbb{P}\left( V_n(Z_0) < 1 + \frac{\varepsilon}{2(\sqrt{e}-\varepsilon)} \text{ and } V_{n-1}(\overline{Z}_0 \cap L) > \sqrt{e}-\varepsilon\right)\\
& = \mathbb{P}\left( V_n(Z_0) < 1 + \frac{\varepsilon}{2(\sqrt{e}-\varepsilon)} \text{ and } V_{n-1}(Z_0 \cap L) > (\sqrt{e}-\varepsilon)V_n(Z_0)^{\frac{n-1}{n}}\right)\\
& \geq \mathbb{P}\left( V_n(Z_0) < 1 + \frac{\varepsilon}{2(\sqrt{e}-\varepsilon)} \text{ and } V_{n-1}(Z_0 \cap L) > (\sqrt{e}-\varepsilon)\left(1+\frac{\varepsilon}{2(\sqrt{e}-\varepsilon)}\right)^{\frac{n-1}{n}}\right)\\
& \geq \mathbb{P}\left( V_n(Z_0) < 1 + \frac{\varepsilon}{2(\sqrt{e}-\varepsilon)} \text{ and } V_{n-1}(Z_0 \cap L) > \sqrt{e}-\frac{\varepsilon}{2}\right)\,.
\end{align*}
Considering now the complement of the event in the last line and using the Chebychev inequality, we obtain
\begin{align*}
&\mathbb{P}\left(V_{n-1}\left(\overline{Z}_0 \cap L\right) > \sqrt{e}-\varepsilon\right)\\
& \geq 1 - \mathbb{P}\left(|V_n(Z_0) - \mathbb{E}V_n(Z_0)|\geq \frac{\varepsilon}{2(\sqrt{e}-\varepsilon)}\right) - \mathbb{P}\left(|V_{n-1}(Z_0 \cap L) - \mathbb{E}V_{n-1}(Z_0 \cap L)| \geq \frac{\varepsilon}{4} \right)\\
& \geq 1 - \varepsilon^{-2}\left(4(\sqrt{e}-\varepsilon)^2\,\var(V_n(Z_0))-16\, \var(V_{n-1}(Z_0 \cap L))\right)\,.
\end{align*}
The variance estimate from Corollary \ref{cor:AsymptotikVolumeSection} then yields an upper bound for $\mathbb{P}\big(V_{n-1}(\overline{Z}_0 \cap L) > \sqrt{e}-\varepsilon\big)$ and the limiting relation follows directly from this upper bound.
\end{proof}

\subsection{Proofs for Section \ref{subsec:HighDimensions}}

\begin{proof}[Proof of Theorems \ref{thm:IsoperimetricRatio}]

Theorem \ref{thm:Skelett} with $\ell=1$ provides an expression for
$$\EE V_{n-1}(Z_0)=\frac{1}{2}\EE\cH^{n-1}({\rm skel}_{n-1}(Z_0))$$
and the mean volume of $Z_0$ is given by \eqref{eq:MeanVolumeZ0}. The result then follows by an application of Stirling's formula.
\end{proof}

\begin{proof}[Proof of Theorem \ref{thm:HighDimensionsVertices}]

From Corollary \ref{cor:EstimateFVector} we obtain for any fixed $\ell\in\N_0$ and $n>l$ that
\begin{equation}\label{eqa}
\sqrt[n]{\mathbb{E} f_\ell(Z_0)}\le  n^{\frac{\ell}{n}}\sqrt[n]{\mathbb{E} f_0(Z_0)}\,.
\end{equation}
To derive a lower bound, we use the lower bound from \eqref{fnml} in Corollary \ref{cor:EndabschaetzungFVektor} to get
$$
\mathbb{E}f_\ell(Z_0)\ge \frac{\kappa_r}{n-\ell}\frac{\omega_{\ell+1}}{\omega_{\ell+1+r}}
\left(\frac{\omega_{r+1}}{\kappa_r}\right)^{n-\ell}\,,
$$
which holds for all $\ell\in\N_0$ and $n>\ell$. Hence,
$$
\liminf\limits_{n\to \infty}\sqrt[n]{\mathbb{E}f_\ell(Z_0)}\ge \liminf\limits_{n\to \infty}\left\{\left(\frac{\kappa_r}{\omega_{\ell+1+r}}\right)^{\frac{1}{n}}
\left(\frac{\omega_{r+1}}{\kappa_{r}}\right)^{1-\frac{\ell}{n}}\right\}\,.
$$
Subsequently, we use that
$$
\left(\frac{\kappa_r}{\omega_{\ell+1+r}}\right)^{\frac{1}{n}}\sim \left(\frac{\Gamma\left(\frac{\ell+1+r}{2}\right)}{\Gamma\left(\frac{r+2}{2}\right)}\right)^{\frac{1}{n}}
\qquad \text{and}\qquad
\frac{\omega_{r+1}}{\kappa_{r}}=2\sqrt{\pi}\,\frac{\Gamma\left(\frac{r+2}{2}\right)}{\Gamma\left(\frac{r+1}{2}\right)}\,.
$$
We distinguish three cases.

\medskip

\noindent
(i) Let $\alpha<0$. Then $r=b\,n^\alpha\to 0$ as $n\to\infty$, and \eqref{fnull} in Corollary \ref{cor:EndabschaetzungFVektor} implies that
\begin{equation}\label{eqb}
\lim\limits_{n\to\infty}\sqrt[n]{\mathbb{E}f_0(Z_0)}=\lim_{n\to\infty}\left(\frac{\omega_{r+1}}{\kappa_r}\right)=\frac{\omega_1}{\kappa_0}=2\,.
\end{equation}
Since
\begin{equation}\label{eqc}
\liminf\limits_{n\to\infty}\sqrt[n]{\mathbb{E}f_\ell(Z_0)}\ge \liminf\limits_{n\to\infty}\left(\frac{2\sqrt{\pi}}{\Gamma\left(\frac{1}{2}\right)}\right)^{1-\frac{\ell}{n}}=2\,,
\end{equation}
the assertion follows from \eqref{eqa}, \eqref{eqb} and \eqref{eqc}.

\medskip

\noindent
(ii) Let $\alpha=0$. Then $r=b$ is independent of $n$ and \eqref{fnull} in Corollary \ref{cor:EndabschaetzungFVektor} yields

\begin{equation}\label{eqd}\lim\limits_{n\to\infty}\sqrt[n]{\EE f_0(Z_0)}=\frac{\o_{b+1}}{\k_b}\,.
\end{equation}
In this case, we conclude that
\begin{equation}\label{eqe}\liminf\limits_{n\to\infty}\sqrt[n]{\EE f_\ell(Z_0)}\geq \liminf\limits_{n\to\infty} \left( \frac{\o_{b+1}}{\k_b}\right)^{1-\frac{\ell}{n}}=\frac{\o_{b+1}}{\k_b}\,,
\end{equation}
hence the assertion follows from \eqref{eqa}, \eqref{eqd} and \eqref{eqe}.

\medskip

\noindent
(iii) Let $\alpha>0$. Then $r=b\,n^\alpha\to\infty$ as $n\to\infty$. We use Stirling's formula to get
\begin{align}
\lim\limits_{n\to\infty} n^{-\frac{\alpha}{2}} \sqrt[n]{\EE f_0(Z_0)}
&
 = \lim\limits_{n\to\infty}n^{-\frac{\alpha}{2}}\left(2 \sqrt{\pi} \frac{\G(\frac{r+2}{2})}{\G(\frac{r+1}{2})}\right)^{\frac{n-1}{n}}
 \nonumber\\
%
%
&= \lim\limits_{n\to\infty}n^{-\frac{\alpha}{2}}\left( \frac{\sqrt{2\pi}}{\sqrt{e}} \sqrt{r+1} \left( \frac{r+2}{r+1}\right)^{\frac{r+1}{2}}\right)^{\frac{n-1}{n}}\nonumber\\
& = \sqrt{2\pi}\lim\limits_{n\to\infty}\left( \frac{\sqrt{b}}{\sqrt{e}}\left( 1+\frac{1}{r+1}\right)^{\frac{r+1}{2}}\right)^{\frac{n-1}{n}}
= \sqrt{2\pi b}\,.\label{eqf}
\end{align}

Repeating the preceding calculations and using again Stirling's formula, we get
\begin{align}
\liminf\limits_{n\to\infty}n^{-\frac{\a}{2}}\sqrt[n]{\EE f_\ell(Z_0)}
&\geq \liminf\limits_{n\to\infty}n^{-\frac{\a}{2}}\left( \frac{\Gamma(\frac{\ell+1+r}{2})}{\frac{r+2}{2}}\right)^{\frac{1}{n}}\left( 2\sqrt{\pi}\frac{\Gamma(\frac{r+2}{2})}{\Gamma(\frac{r+1}{2})}\right)^{1-\frac{\ell}{n}}\nonumber\\
&= \liminf\limits_{n\to\infty}\left( \frac{\Gamma(\frac{\ell+1+r}{2})}{\Gamma(\frac{r+2}{2})}\right)^{\frac{1}{n}}\sqrt{2\pi b}\nonumber\\
&=\liminf\limits_{n\to\infty}\left( \frac{\sqrt{\frac{r+2}{2}}}{\sqrt{\frac{\ell+1+r}{2}}}\right)^{\frac{1}{n}} \frac{\left( \frac{\ell+1+r}{2e}\right)^{\frac{\ell+1+r}{2n}}}{\left( \frac{r+2}{2e}\right)^{\frac{r+2}{2n}}}\sqrt{2\pi b}\nonumber\\
&=\liminf\limits_{n\to\infty}\frac{(\ell+1+r)^{\frac{\ell+1+r}{2n}}}{(r+2)^{\frac{r+2}{2n}}} \sqrt{2\pi b}\nonumber\\
&=\liminf\limits_{n\to\infty}\left(1+\frac{\ell-1}{r+2}\right)^{\frac{r+2}{2n}}\left(\ell+1+r\right)^{\frac{\ell-1}{2n}}\sqrt{2\pi b}
=\sqrt{2\pi b}\,.\label{eqg}
\end{align}
Hence, the assertion is implied by \eqref{eqa}, \eqref{eqf} and \eqref{eqg}.
\end{proof}

\begin{proof}[Proof of Theorem \ref{thm:HighDimensionRfest}]

(i) If $\a\leq 0$, then Corollary \ref{cor:EstimateFVector} and Theorem \ref{thm:HighDimensionsVertices} imply that
\[\limsup\limits_{n\to\infty}\sqrt[n]{\EE f_{n-\ell}(Z_0)}\leq \limsup\limits_{n\to\infty}\sqrt[n]{\EE f_0(Z_0)}
=\begin{cases}2& : \a<0\\ \frac{\o_{b+1}}{\k_b}& : \a=0\,.\end{cases}\]

\noindent
(ii) -- (iv) Let $\a>0$. Then $r=b\, n^\alpha\to\infty$ as $n\to\infty$ and the upper bound follows again from  Corollary \ref{cor:EstimateFVector} and from Theorem \ref{thm:HighDimensionsVertices}. Let $\ell\in\N$ be fixed. For the lower bounds, we use \eqref{fnml} in Corollary \ref{cor:EndabschaetzungFVektor} to get for an arbitrary $\beta>0$ (which will be specified later) that
\begin{align*}
\liminf\limits_{n\to\infty}n^{-\frac{\beta}{2}}\sqrt[n]{\EE f_{n-\ell}(Z_0)}&\geq \liminf\limits_{n\to\infty}n^{-\frac{\beta}{2}}\left(\k_r \frac{\o_{n-\ell+1}}{\o_{n-\ell+1+r}} \left( \frac{\o_{r+1}}{\k_r}\right)^\ell \right)^{\frac{1}{n}}\\
&=\liminf\limits_{n\to\infty}n^{-\frac{\beta}{2}}\left(\frac{\Gamma(\frac{n-\ell+1+r}{2})} {\Gamma(\frac{r+2}{2}) \Gamma(\frac{n-\ell+1}{2})} \right)^{\frac{1}{n}} \left( 2\sqrt{\pi}\frac{\Gamma(\frac{r+2}{2})}{\Gamma(\frac{r+1}{2})}\right)^{\frac{\ell}{n}}\\
&=\liminf\limits_{n\to\infty}n^{-\frac{\beta}{2}}\frac{(n-\ell+1+r)^{\frac{n+r}{2n}}}
{(r+2)^{\frac{r}{2n}} (n-\ell+1)^{\frac{1}{2}}},
\end{align*}
where we used Stirling's formula and basic asymptotic relations as before.
Proceeding from this, we get
\begin{align*}
\liminf\limits_{n\to\infty}n^{-\frac{\beta}{2}}\sqrt[n]{\EE f_{n-\ell}(Z_0)}
&\geq \liminf\limits_{n\to\infty}n^{-\frac{\beta}{2}} \frac{(n+r)^{\frac{n+r}{2n}}}{r^{\frac{r}{2n}}\sqrt{n}}
 =\liminf\limits_{n\to\infty}n^{-\frac{\beta}{2}} \frac{\sqrt{n+r}}{\sqrt{n}}\left( \frac{n+r}{r}\right)^{\frac{r}{2n}}\\
&= \liminf\limits_{n\to\infty}n^{-\frac{\beta}{2}}\sqrt{1+\frac{r}{n}}\left(1+\frac{n}{r}\right)^{\frac{r}{2n}}.
\end{align*}
If $\alpha\in (0,1)$,  we choose $\beta=0$ and have
\[\liminf\limits_{n\to\infty}\sqrt[n]{\EE f_{n-\ell}(Z_0)}
\geq \liminf\limits_{n\to\infty}\sqrt{1+b n^{\alpha-1}} \left( 1+\frac{n^{1-\alpha}}{b}\right)^{\frac{b}{2 n^{1-\alpha}}}=1.\]
If $\alpha=1$, we choose $\beta=0$ and conclude
\begin{align*}
\liminf\limits_{n\to\infty}\sqrt[n]{\EE f_{n-\ell}(Z_0)}
\geq \sqrt{1+b}\left( 1+\frac{1}{b}\right)^{\frac{b}{2}}.
\end{align*}
If $\alpha>1$,  we choose $\beta=\alpha-1$ and get
\begin{align*}
\liminf\limits_{n\to\infty}n^{\frac{1-\a}{2}}\sqrt[n]{\EE f_{n-\ell}(Z_0)}
&\geq \liminf\limits_{n\to\infty}n^{\frac{1-\a}{2}}\sqrt{1+b n^{\alpha-1}} \left(1+\frac{1}{bn^{\alpha-1}}\right)^{\frac{bn^{\a-1}}{2}}\\
&= \liminf\limits_{n\to\infty}\sqrt{n^{1-\a}+b}\sqrt{e}=\sqrt{be},
\end{align*}
which proves the theorem in all cases.
\end{proof}

\begin{proof}[Proof of Theorem \ref{thm:HighDimensionsR=bn}]
We omit the proof, since it is similar to the arguments for the preceding two theorems.
\end{proof}

\subsection*{Acknowledgement}
The authors would like to thank an anonymous referee for his useful comments which helped to improve the manuscript.
JH and DH have been supported by the German Research Foundation via the Research Group ``Geometry and Physics of Spatial Random Systems''. CT has been supported by the German Research Foundation (DFG) via SFB-TR 12 ``Symmetries and Universality in Mesoscopic Systems''.

\end{document}